\documentclass[11pt,reqno]{amsart}

\usepackage{microtype}

\usepackage[T1]{fontenc}

\usepackage{amsmath}								
\usepackage{amssymb}
\usepackage{amsthm}
\usepackage{amscd}
\usepackage{amsfonts}
\usepackage{stmaryrd}

\usepackage{euler}

\usepackage{extarrows}

\usepackage[backref, colorlinks, linktocpage, citecolor = magenta, linkcolor = blue]{hyperref}
\usepackage{color}
\usepackage[backrefs]{amsrefs}

\usepackage{tikz}									
\usetikzlibrary{matrix}
\usetikzlibrary{patterns}
\usetikzlibrary{matrix}
\usetikzlibrary{positioning}
\usetikzlibrary{decorations.pathmorphing}
\usetikzlibrary{cd}

\usepackage{fullpage}
\usepackage{enumerate}

\newtheorem{theorem}{Theorem}[section]
\newtheorem{lemma}[theorem]{Lemma}
\newtheorem{proposition}[theorem]{Proposition}
\newtheorem{corollary}[theorem]{Corollary} 

\theoremstyle{definition}
\newtheorem{definition}[theorem]{Definition}
\newtheorem{example}[theorem]{Example}

\newtheorem{remark}[theorem]{Remark}

\newcommand{\marg}[1]{\null}

\newcommand{\marge}[1]{\null}

\newcommand{\R}{\mathbb{R}}
\newcommand{\Rbar}{\overline{\mathbb{R}}}
\newcommand{\Z}{\mathbb{Z}}

\newcommand{\N}{\mathbb{N}}
\newcommand{\PP}{\mathbb{P}}

\newcommand{\A}{\mathbb{A}}
\newcommand{\G}{\mathbb{G}}

\newcommand{\Deltabar}{\overline{\Delta}}
\newcommand{\phibar}{\overline{\phi}}
\newcommand{\Sigmabar}{\overline{\Sigma}}
\newcommand{\sigmabar}{\overline{\sigma}}
\newcommand{\Mbar}{\overline{M}}
\newcommand{\calObar}{\overline{\mathcal{O}}}
\newcommand{\calA}{\mathcal{A}}
\newcommand{\calB}{\mathcal{B}}

\newcommand{\calF}{\mathcal{F}}

\newcommand{\calH}{\mathcal{H}}
\newcommand{\calM}{\mathcal{M}}
\newcommand{\calO}{\mathcal{O}}

\newcommand{\calX}{\mathcal{X}}
\newcommand{\calY}{\mathcal{Y}}
\newcommand{\calMbar}{\overline{\mathcal{M}}}
\newcommand{\frakS}{\mathfrak{S}}

\newcommand{\frakp}{\mathfrak{p}}

\newcommand{\bfp}{\mathbf{p}}

\DeclareMathOperator{\Spec}{Spec}
\DeclareMathOperator{\Hom}{Hom}

\DeclareMathOperator{\Aut}{Aut}

\DeclareMathOperator{\Isom}{Isom}
\DeclareMathOperator{\trop}{trop}
\DeclareMathOperator{\val}{val}

\DeclareMathOperator{\LOG}{LOG}
\DeclareMathOperator{\FAN}{FAN}
\DeclareMathOperator{\pr}{pr}

\DeclareMathOperator{\HOM}{HOM}

\title[Non-Archimedean geometry of Artin fans]{Non-Archimedean geometry of Artin fans}

\author{Martin Ulirsch}
\address{Institut f\"ur Mathematik\\
Goethe-Universit\"at Frankfurt am Main\\
Robert-Mayer-Str. 6-8\\
60325 Frankfurt am Main\\
Germany}
\email{ulirsch@math.uni-frankfurt.de}

\subjclass[2010]{14T05; 14A20; 32P05}
\date{\today}
\thanks{M.U.'s research was supported in part by funds from BSF grant 201025, NSF grants DMS0901278 and DMS1162367, and by the SFB/TR 45 'Periods, Moduli Spaces and Arithmetic of Algebraic Varieties' of the DFG (German Research Foundation) as well as by the Hausdorff Center for Mathematics at the University of Bonn.} 

\begin{document}

\maketitle

\begin{abstract}
The purpose of this article is to study the role of Artin fans in tropical and non-Archimedean geometry. Artin fans are logarithmic algebraic stacks that can be described completely in terms of combinatorial objects, so called Kato stacks, a stack-theoretic generalization of K. Kato's notion of a fan. Every logarithmic algebraic stack admits a tautological strict morphism $\phi_\mathcal{X}:\mathcal{X}\rightarrow\mathcal{A}_\mathcal{X}$ to an associated Artin fan. The main result of this article is that, on the level of underlying topological spaces, the natural functorial tropicalization map of $\mathcal{X}$ is nothing but the non-Archimedean analytic map associated to $\phi_\mathcal{X}$ by applying Thuillier's generic fiber functor. Using this framework, we give a reinterpretation of the main result of Abramovich-Caporaso-Payne identifying the moduli space of tropical curves with the non-Archimedean skeleton of the corresponding algebraic moduli space. 
\end{abstract}

\setcounter{tocdepth}{1}
\tableofcontents


\section{Introduction}

Let $k$ be an algebraically closed field that is endowed with the trivial absolute value. Inspired by the construction of the non-Archimedean skeleton associated to a toroidal embedding in \cite{Thuillier_toroidal} and recent advances in the tropical geometry of moduli spaces in \cite{AbramovichCaporasoPayne_tropicalmoduli}, the author, in  \cite{Ulirsch_functroplogsch}, has defined  a \emph{tropicalization map} $\trop_X\mathrel{\mathop:}X^\beth\rightarrow \Sigmabar_X$ associated to every fine and saturated logarithmic scheme locally of finite type over $k$  that generalizes earlier constructions for split algebraic tori \cite{EinsiedlerKapranovLind_amoebas, Gubler_tropvar, Gubler_guide} and toric varieties \cite{Kajiwara_troptoric, Payne_anallimittrop}. 

The heuristic that guides this construction is based on the fact that in the special case that $X$ has \emph{no monodromy}, there is a natural strict \emph{characteristic morphism} into a sharp monoidal space $F_X$ that captures the combinatorics of the logarithmic strata of $X$, a so-called \emph{Kato fan}, as introduced in \cite{Kato_toricsing}. The tropicalization map $\trop_X$ is simply the "analytification" of this characteristic morphism in a suitable topological sense. In Section \ref{section_functroplogstack} below we rephrase this construction in a way that immediately generalizes to logarithmic stacks.

The advantage and the beauty of working with Kato fans mostly stems from their dual nature as algebraic objects, whose geometry is completely determined by combinatorics. Their usage, however, suffers from two major inconveniences:
\begin{enumerate}
\item\label{problemI} Not every fine and saturated logarithmic scheme admits a characteristic morphism into a Kato fan. In fact, as soon as the logarithmic structure on $X$ encodes \emph{monodromy}, a Kato fan does not exist (see \cite[ Example 4.8]{Ulirsch_functroplogsch}). 
\item\label{problemII} Kato fans live in the category of sharp monoidal spaces. So, despite their inherently algebraic nature, most of the techniques of algebraic geometry cannot directly be applied to Kato fans.
\end{enumerate}

Both problems can be resolved by working with \emph{Artin fans}, a notion that has originally been introduced in \cite{AbramovichWise_invlogGromovWitten, AbramovichChenMarcusWise_boundedness} in the context of logarithmic Gromov-Witten theory, but can be implicitly traced back to the work of Olsson \cite{Olsson_loggeometryalgstacks} on classifying stacks of logarithmic structures. Similiarly to Kato fans, Artin fans can be described completely in terms of combinatorics, since every Artin fan is \'etale locally isomorphic to a stack quotient $\big[X\big/T\big]$, where $X$ is a $T$-toric variety. In fact, in \cite{CavalieriChanUlirschWise_tropstack}*{Theorem 3} (also see Theorem \ref{thm_Artinfan=Katostack} below) it is shown that the $2$-category of Artin fans is equivalent to the category of \emph{Kato stacks}, i.e. to a category of geometric stacks over the category of Kato fans (see Section \ref{section_Katostacks} below).

\subsection{Non-Archimedean geometry of Artin fans} In Section \ref{section_logstacktoArtinfan} below we will see how to naturally associate to a fine and saturated logarithmic stack $\calX$ that is locally of finite type over $k$ an Artin fan $\calA_\calX$ with faithful monodromy that parametrizes the logarithmic strata of $\calX$ (as well as a certain representation of their generic automorphisms) together with a tautological strict morphism $\phi_\calX\colon \calX\rightarrow \calA_\calX$ that is functorial with respect to strict morphisms. 

In Section \ref{section_functroplogstack} we generalize the construction from \cite{Ulirsch_functroplogsch} to a natural continuous tropicalization map $\trop_\calX\colon \vert \calX^\beth\vert \rightarrow \Sigmabar_\calX$. Expanding on the heuristics from \cite{Ulirsch_functroplogsch} that $\trop_\calX$ is the "analytification" of the characteristic morphism, it seems natural to expect that, after applying Thuillier's \cite{Thuillier_toroidal} generic fiber functor $(.)^\beth$ (see Section \ref{section_stackygenericfiber} below), the non-Archimedean analytic map $\phi_\calX^\beth\mathrel{\mathop:}\calX^\beth\rightarrow\calA_\calX^\beth$ should, at least topologically, agree with $\trop_\calX$. The following Theorem \ref{thm_trop=anal} makes our heuristic precise. 

\begin{theorem}\label{thm_trop=anal}
 Let $\calX$ be a fine and saturated logarithmic stack locally of finite type over $k$. There is a natural homeomorphism $\mu_\calX\mathrel{\mathop:}\big\vert\calA_\calX^\beth\big\vert\xlongrightarrow{\sim}\Sigmabar_\calX$ that makes the diagram
\begin{center}\begin{tikzcd}
 \big\vert\calX^\beth\big\vert \arrow[rrd,bend left,"\trop_\calX"] \arrow[rd,"\phi_\calX^\beth"']&&\\
 &\big\vert\calA_\calX^\beth\big\vert \arrow[r,"\sim","\mu_\calX"'] & \Sigmabar_\calX
\end{tikzcd}\end{center}
commute. 
\end{theorem}

Suppose now that $\calX$ is a logarithmically smooth Deligne-Mumford stack. By \cite[Theorem 1.2]{Ulirsch_functroplogsch} and \cite{Thuillier_toroidal} (also see \cite{AbramovichCaporasoPayne_tropicalmoduli}*{Section 5} and Proposition \ref{prop_functroplogstack} (iii) below) the tropicalization map $\trop_\calX$ admits a continuous section $J_\calX\mathrel{\mathop:}\Sigmabar_\calX\rightarrow \big\vert\calX^\beth\big\vert$ such that the composition $\bfp_\calX=J_\calX\circ\trop_\calX$ is a strong deformation retraction of $\big\vert\calX^\beth\big\vert$ onto the \emph{skeleton} $\frakS(\calX)$ associated to $\calX$. Therefore Theorem \ref{thm_trop=anal} immediately implies the following Corollary \ref{cor_skel=anal}.

\begin{corollary}\label{cor_skel=anal}
If $\calX$ is a logarithmically smooth Deligne-Mumford stack, there is a natural homeomorphism $\widetilde{\mu}_\calX\mathrel{\mathop:}\big\vert\calA_\calX^\beth\big\vert\xlongrightarrow{\sim}\frakS(\calX)$ that makes the diagram
\begin{center}\begin{tikzcd}
 \big\vert\calX^\beth\big\vert \arrow[drr,bend left,"\bfp_\calX"] \arrow[dr,"\phi_\calX^\beth"']&& \\
 &\big\vert\calA_\calX^\beth\big\vert \arrow[r,"\sim","\widetilde{\mu}_\calX"'] & \frakS(\calX) 
\end{tikzcd}\end{center}
commute. 
\end{corollary}

Theorem \ref{thm_trop=anal} and Corollary \ref{cor_skel=anal} in particular show that both the generalized extended complex and the skeleton of a logarithmic stack canonically carry the structure of an analytic stack. The author expects that this is only the starting point of a longer program and that one can endow the different kinds of non-Archimedean skeletons \cite{Berkovich_book, Berkovich_localcontractibility, BakerPayneRabinoff_nonArchtrop, GublerRabinoffWerner_skeletons&tropicalizations, MustataNicaise_KontsevichSoibelman} with the structure of an analytic stack and thereby categorifying these a priori only topological constructions within the realm of analytic geometry.  

Conversely, Theorem \ref{thm_trop=anal} also allows us to naturally endow the topological space $\big\vert \calA_\calX^\beth\big\vert$ with the structure of a generalized extended cone complex that correctly encodes the combinatorics of the logarithmic strata, a point of view that allows us to recover the main result of \cite{AbramovichCaporasoPayne_tropicalmoduli} on the tropical geometry of moduli spaces.

\subsection{Artin fans and moduli spaces}
In \cite{AbramovichCaporasoPayne_tropicalmoduli} the authors develop a framework to understand the tropical geometry of moduli spaces (with a focus on the moduli space of stable curves) that is based on Thuillier's non-Archimedean skeleton of a toroidal embedding \cite{Thuillier_toroidal}. Their recipe has already been applied in a multitude of other situations \cite{CavalieriMarkwigRanganathan_admissiblecovers, CavalieriHampeMarkwigRanganathan_tropHassett, Ulirsch_tropHassett, Ranganathan_ratcurvesontorvars}. We now revisit their result, equipped with the techniques of this paper and the theory of \emph{tropical moduli stacks} that the author and his collaborators have developed in \cite{CavalieriChanUlirschWise_tropstack}.

Let $2g-2+n>0$. The moduli stack $\calM_{g,n}^{trop}$ is the unique Kato stack, whose fiber over an affine Kato fan $\Spec P$ (for a sharp monoid $P$) is the groupoid of stable tropical curves of genus $g$ with $n$ marked legs and edge lengths in $P$. As explained in \cite{CavalieriChanUlirschWise_tropstack}*{Section 7}, the Artin fan associated to $\calM_{g,n}^{trop}$ is a stack $\calA_{\calM_{g,n}^{trop}}$ over the category of logarithmic schemes, whose fiber over a logarithmic scheme $S$ consists of collections $(\Gamma_s)$ of tropical curves in $\calM_{g,n}^{trop}\big(\Spec \Mbar_{S,s}\big)$ that are parametrized by the geometric points $s$ of $S$ and are compatible with respect to \'etale specialization. Moreover, there is a natural strict logarithmic \emph{tropicalization morphism} $\trop_{g,n}\colon \calM_{g,n}^{log}\rightarrow \calA_{\calM_{g,n}^{trop}}$ from the moduli stack of logarithmic curves (in the sense of \cite{Kato_logsmoothcurves}) that is given by associating to a logarithmic curve $X\rightarrow S$ the collection of dual tropical curves $(\Gamma_{X_s})$ for all geometric points $s$ of $S$. 

 The moduli stack $\calM_{g,n}^{trop}$ is a stack-theoretic generalization of the set-theoretic (coarse) tropical moduli space $M_{g,n}^{trop}$, a generalized cone complex whose points naturally parametrize the isomorphism classes of stable tropical curves of genus $g$ with $n$ marked legs. Its canonical extension $\Mbar_{g,n}^{trop}$ naturally parametrizes so-called \emph{extended tropical curves} which allow edges of infinite length. We refer the reader to \cite{HuszarMarcusUlirsch_troplogclutch&glue} for a stack-theoretic generalization of this moduli space as a functor over the category of  \emph{pointed monoids} that contain an absorbing element $\infty$.

In \cite{AbramovichCaporasoPayne_tropicalmoduli} the authors identify $\Mbar_{g,n}^{trop}$ with the non-Archimedean skeleton of $\calMbar_{g,n}$, thought of as a toroidal embedding with respect to the Deligne-Knudsen-Mumford boundary. Moreover, they show that the natural retraction map $\bfp_{g,n}$ onto the skeleton $\frakS\big(\calMbar_{g,n}\big)$ is exactly the tropicalization map of $\calMbar_{g,n}$ that associates to a stable degeneration of an algebraic curve the dual tropical curve of the special fiber. 

The generalized extended cone complex associated to $\calM_{g,n}^{trop}$ is nothing but the set-theoretic moduli space $\Mbar_{g,n}^{trop}$ of extended stable tropical curves considered in \cite{AbramovichCaporasoPayne_tropicalmoduli}, which, by Proposition \ref{prop_anal=conecomplex} (i.e. essentially by Theorem \ref{thm_trop=anal} above) is homeomorphic to $\big\vert \calA_{\calM_{g,n}^{trop}}^\beth\big\vert$ (also see \cite[Corollary 1]{CavalieriChanUlirschWise_tropstack}). 
The following Theorem \ref{thm_ACPArtin} rephrases the main result of \cite{AbramovichCaporasoPayne_tropicalmoduli} in the language developed in this article. 

\begin{theorem}\label{thm_ACPArtin}
The Artin fan of $\calA_{\calM_{g,n}^{trop}}$ (as a logarithmic stack) is naturally isomorphic to the Artin fan $\calA_{\calM_{g,n}^{log}}$ of $\calM_{g,n}^{log}$ and the tautological strict morphism 
\begin{equation*}
\phi_{g,n}\colon \calA_{\calM_{g,n}^{trop}}\longrightarrow \calA_{\calM_{g,n}^{log}}
\end{equation*}
induces an isomorphism $M_{g,n}^{trop}\xrightarrow{\sim} \Sigma_{\calM_{g,n}^{log}}$ of generalized cone complexes that makes the induced diagram
\begin{center}\begin{tikzcd}
\big\vert\calMbar_{g,n}^\beth\big\vert \arrow[dr,"\trop_{g,n}^\beth"'] \arrow[drr, bend left, "\phi_{\calM_{g,n}^{log}}^\beth"]&& \\
&\Mbar_{g,n}^{trop} \arrow[r,"\sim","\phi_{g,n}^\beth"'] & \Sigmabar_{\calM_{g,n}^{log}}
\end{tikzcd}\end{center}
commute.
\end{theorem}

Although the generalized extended cone complexes $\Mbar_{g,n}^{trop}$ and $\Sigmabar_{\calM_{g,n}^{log}}$ are isomorphic, the strict morphism $\phi_{g,n}\colon \calA_{\calM_{g,n}^{trop}}\rightarrow \calA_{\calM_{g,n}^{log}}$ is not an isomorphism of logarithmic stacks. In other words, the Artin fan associated to the Artin fan $\calA_{\calM_{g,n}^{trop}}$ (thought of as a logarithmic stack) is not $\calA_{\calM_{g,n}^{trop}}$ itself, but rather $\calA_{\calM_{g,n}^{log}}$, since $\calA_{\calM_{g,n}^{trop}}$ does not have faithful monodromy. 

In fact, the points of both $\calA_{\calM_{g,n}^{trop}}$ and $\calA_{\calM_{g,n}^{log}}$ are the same (and parametrized by the category of stable finite vertex-weighted graphs of genus $g$ with $n$ legs). The crucial difference is that the automorphism group of a point in $\calA_{\calM_{g,n}^{trop}}$ is $\Aut(G)$ while the automorphisms of its image in $\calA_{\calM_{g,n}^{log}}$ is the image of $\Aut(G)$ under the natural monodromy representation (see Section \ref{section_logstacktoArtinfan} below for details and, in particular, Example \ref{example_M11}).




\subsection{Complements and Applications}

\subsubsection{The case of toric varieties} 
Let $T$ be a split algebraic torus with character lattice $M$ and cocharacter lattice $N$. Let $X=X(\Delta)$ be a $T$-toric variety defined by a rational polyhedral fan. In \cite{Kajiwara_troptoric} and \cite{Payne_anallimittrop} the authors independently construct a partial compactification $N_\R(\Delta)$ of $N_\R=N\otimes\R$ that allows them to define a natural continuous tropicalization map $\trop_\Delta\mathrel{\mathop:}X^{an}\rightarrow N_\R(\Delta)$. 

By \cite[Proposition 7.1]{Ulirsch_functroplogsch}, the restriction of the Kajiwara-Payne tropicalization map to $X^\beth\subseteq X^{an}$ is nothing but the tropicalization $\trop_\Delta\colon X^\beth\rightarrow \Deltabar$ defined in \cite{Ulirsch_functroplogsch}, where $\Deltabar$ denotes the closure of $\Delta$ in $N_\R(\Delta)$. Therefore there is a natural homeomorphism $\mu_X\mathrel{\mathop:}\big\vert\calA_X^\beth\big\vert\xlongrightarrow{\sim}\Deltabar$ that makes the diagram
\begin{center}\begin{tikzcd}
X^\beth \arrow[dr,"\phi_X^\beth"']\arrow[drr, bend left, "\trop_{\Delta}"]&&\\
& \calA_X \arrow[r,"\sim","\mu_X"'] & \Deltabar
\end{tikzcd}\end{center}
commute.

Since $\calA_X=\big[X\big/T\big]$ by Example \ref{example_toricArtinfan} below, this suggests a new perspective on the tropicalization of toric varieties as the non-Archimedean analytic stack quotient
\begin{equation*}
X^\beth\longrightarrow \big[X^\beth\big/T^\beth\big] \ .
\end{equation*}
With some technical modifications this result can be generalized to all toric varieties defined over any non-Archimedean field, not necessarily carrying the trivial absolute value (see \cite[Theorem 1.1]{Ulirsch_trop=quot}). As explained in \cite{Ulirsch_trop=quot}*{Section 1}, this adds a further layer to the analogy between the tropicalization map in the non-Archimedean and the moment map in the Archimedean world.

\subsubsection{Applications}
Artin fans have been introduced in \cite{AbramovichWise_invlogGromovWitten, AbramovichChenMarcusWise_boundedness} in order to study the toroidal birational geometry of the moduli space of logarithmically stable maps. The main point hereby is that they form a convenient and flexible framework to study toroidal modifications via subdivisions, expanding on K. Kato's construction via Kato fans in \cite[Section 10]{Kato_toricsing} as well as the original approach for toroidal embeddings in \cite[Section II.2]{KKMSD_toroidal}.

Theorem \ref{thm_trop=anal} and its toric variant in \cite{Ulirsch_trop=quot} have already found applications to tropical geometry. In \cite{Ranganathan_Artinfanrealizability} Ranganathan shows that, while not every tropical curve $\Gamma\subseteq \R^n$ is realizable by a curve in a $T\simeq\G^n$-toric variety $X$, it is possible to always find a curve over the Artin fan $\calA_X=\big[X\big/T\big]$ that tropicalizes to $\Gamma$. The main reason why this approach works can be traced back to the vanishing of certain cohomological logarithmic obstruction groups over $\calA_X$ (see \cite[Section 3]{CheungFantiniParkUlirsch_faithfulrealizability} for a detailed discussion). 
Furthermore, Artin fans in general and Theorem \ref{thm_trop=anal} in particular form a crucial technical ingredient in the following papers:
\begin{itemize}
\item in Ranganathan's \cite{Ranganathan_ratcurvesontorvars} conceptual explanation of the classical correspondence theorem between rational curves on toric varieties and their tropical counterparts \cite{Mikhalkin_enumerativetrop, NishinouSiebert_toricdeg&tropcurves} in terms of logarithmic Gromov-Witten theory and via the tropical geometry of moduli spaces \cite{AbramovichCaporasoPayne_tropicalmoduli};
\item in Ranganathan's \cite{Ranganathan_superabundantgeometries} study of superabundant geometries and the realizability of tropical curves;
\item and in Jensen and Ranganathan's \cite{JensenRanganathan_BrillNoetherfixedgonality} determination of the dimension of the Brill-Noether spaces of complete linear series for curves of fixed gonality, using methods from logarithmic Gromov-Witten theory. 
\end{itemize}

\subsubsection{Generalizations}

In \cite[Section 13]{Lorscheid_tropicalschemes} Lorscheid provides an enrichment of the structure of the Kato fan $F_X$ of a logarithmic scheme $X$ without monodromy within the category of \emph{ordered blue schemes}. The advantage of his approach is that now one can also endow the tropicalization of a closed subscheme of $X$ with an algebraic structure that is inspired by the theory of \emph{tropical schemes} \cite{GiansiracusaGiansiracusa_tropicalschemes, MaclaganRincon_tropicalideals, MaclaganRincon_tropicalschemes}. The author hopes that eventually there will be a general stack-theoretic object, an "ordered blue stack", which generalizes both ordered blue schemes and Kato stacks, and which categorifies the tropical geometry of subspaces of logarithmic stacks.

In another direction, it would of course be desirable to have a relative theory of Artin fans over arbitrary valuation rings of rank one. The idea in this case would be to have Artin fans that are \'etale locally isomorphic to stack quotients of toric schemes (see \cite[Section IV.3]{KKMSD_toroidal} and \cite[Section 6 and 7]{Gubler_guide}) by their big tori. We  refrain from including this here, since the necessary technical foundations for logarithmic structures over arbitrary valuation rings of rank one and their connections to tropical and non-Archimedean geometry seem to have not appeared in the literature.


\subsection{Conventions}\label{section_conventions}

A \emph{monoid} is a commutive semigroup with a unit element and will be mostly written additively. Throughout this article, unless mentioned otherwise, we assume that all monoids are fine and saturated and implicitly amalgamated sums will be taken in the category of fine and saturated monoids. We say that a monoid $P$ is \emph{sharp} if its group of units $P^\ast$ is trivial.

Let $S\subseteq P$ be a submonoid. The \emph{localization} $S^{-1}P$ of $P$ along $S$ is the unique monoid that factors every monoid homomorphism $f\colon P\rightarrow Q$ for which every $f(s)$ for $s\in S$ is a unit in $Q$. We call the monoid $S^{-1}P/(S^{-1}P)^\ast$ the \emph{sharp localization} of $P$ along $S$; it is the unique monoid that uniquely factors every monoid homomorphism $f\colon P\rightarrow Q$ into a sharp monoid $Q$ that fulfills $f(s)=0$ for all $s\in S$. If $f\in P$ we write $P_f$ for the localization of $P$ along the submonoid generated by $f$ and $P_f/P_f^\ast$ for its sharp localization.

A \emph{sharp monoidal space} is a tuple $(X,\calO_X)$ consisting of a topological space $X$ and a sheaf of monoids $\calO_X$ such that for every point $x\in X$ the unit group $\calO_{X,x}^{\ast}$ is trivial. A morphism $(X,\calO_X)\rightarrow (Y,\calO_Y)$ of sharp monoidal spaces consists of a continuous map $f\colon X\rightarrow Y$ together with a pullback homomorphism $f^\flat\colon f^{-1}\calO_Y\rightarrow \calO_X$ such that for every $x\in X$ the homomorphism $f^\flat_x\colon \calO_{Y.f(y)}\rightarrow\calO_{X,x}$ is a \emph{local}, i.e. we have $(f^\flat)^{-1}(0_{X,x})=\{0_{Y,f(x)}\}$. In a slight but familiar abuse of notation we simply denote the tuple $(f,f^\flat)$ by the letter $f$. 

The letter $k$ will always stand for an algebraically closed field $k$ that is endowed with the trivial absolute value and the trivial logarithmic structure $k^\ast\subseteq k$. We are working in the category $\mathbf{LSch}$ of fine and saturated logarithmic schemes $X=(\underline{X},M_X)$ in the sense of \cite{Kato_logstr}, which are locally of finite type over $k$, unless mentioned otherwise. In particular, we shall simply refer to a fine and saturated logarithmic scheme that is locally of finite type over $k$ as a \emph{logarithmic scheme}. 

Recall that a morphism $f\mathrel{\mathop:}X\rightarrow X'$ of logarithmic schemes is said to be \emph{strict}, if the canonical morphism $f^\ast M_{X'}\rightarrow M_X$ is an isomorphism. The category $\mathbf{LSch}$ is endowed with the  \emph{strict \'etale topology}, the coarsest topology that makes the functor $X\mapsto \underline{X}$ continuous with respect to the \'etale topology.

The term \emph{logarithmic stack} will always refer to a fine and saturated logarithmic stack that is locally of finite type over $k$, i.e. an algebraic stack $\underline{\calX}$ locally of finite type over $k$ that is endowed with a fine and saturated logarithmic structure $M_\calX$ in the lisse-\'etale topology. To every logarithmic stack $\calX$ we may associate a geometric stack over the $\mathbf{LSch}_{str.\'et.}$, which (in a slight abuse of notation) we are also going to denote by $\calX$ (see \cite[Section 4]{Abramovichetal_logmoduli} for a short discussion of this issue).

Fix a logarithmic scheme $S$. In  \cite{Olsson_loggeometryalgstacks} Olsson introduces an algebraic stack  $\underline{\LOG}_S$, whose objects are morphisms $T\rightarrow S$ of logarithmic schemes and whose arrows are strict morphisms over $S$, the \emph{classifying stack of logarithmic structures over $S$}. Given a logarithmic scheme $X$ (or more generally a logarithmic stack) over $S$, there is a natural \emph{tautological morphism} $\underline{X}\rightarrow \underline{\LOG}_S$ determined by associating to $f\in \underline{X}(T)$ the logarithmic scheme $(T,f^\ast M_X)$ over $S$. Conversely, let $\underline{X}$ be a scheme over $\underline{S}$. In this situation, the data of a logarithmic structure on $\underline{X}$ over $S$ is equivalent to the data of the tautological morphism $\underline{X}\rightarrow \underline{\LOG}_S$. It has been shown in \cite[Section 4]{Olsson_loggeometryalgstacks} that a logarithmic morphism $X\rightarrow S$ is logarithmically \'etale (or logarithmically smooth) if and only if the morphism $\underline{X}\rightarrow \underline{\LOG}_S$ is \'etale (or smooth respectively). 

Let $S$ be a scheme (or more generally an algebraic stack). A \emph{geometric point} $s$ of $S$ is a point in the small \'etale site of $S$. So $s$ is an equivalence class of morphism $\Spec K\rightarrow S$, where $K$ is an algebraically closed field and the equivalence relation is generated by identifying $\Spec K\rightarrow S$ and $\Spec L\rightarrow S$, if $L\vert K$ is a field extension. The field extension $L\vert K$ induces an equivalence $(\Spec L)_{et}\simeq(\Spec K)_{et}$ and both are equivalent to $\mathbf{Sets}$. So we may think of a geometric point $s$ as a functor $s^\ast\colon S_{et}\rightarrow\mathbf{Sets}$. An \emph{\'etale specialization} $t\leadsto s$ between geometric points is a natural transformation $s^\ast\rightarrow t^\ast$, i.e. a choice of a point in $t^\ast U$ for every \'etale neighborhood $U$ of $s$ (see \cite{CavalieriChanUlirschWise_tropstack}*{Appendix A} for details). 

\subsection{Acknowledgements}
First and foremost the author would like to thank Dan Abramovich for sharing a multitude of insights in discussions around the topic of this article as well as Jonathan Wise for teaching him about Artin fans during many conversations. 
The author also profited from many conversations with other mathematicians, in particular with Matt Baker, Dori Bejleri, Renzo Cavalieri, Melody Chan, Andreas Gross, Oliver Lorscheid,  Diane MacLagan, Steffen Marcus, Sam Payne, Dhruv Ranganathan, Mattia Talpo, Michael Temkin, Jenia Tevelev, Amaury Thuillier, Annette Werner, and Tony Yue Yu. Many thanks are also due to the anonymous referee(s), who pointed out several technical inaccuracies in an earlier version of this manuscript and gave a multitude of suggestions that significantly improved the exposition in this article. 


\section{Kato stacks}\label{section_Katostacks}


In this section we introduce the theory of \emph{Kato stacks} in order to generalize the older notion of \emph{generalized (rational polyhedral) cone complexes} into a natural $2$-categorical setting (see \cite{AbramovichCaporasoPayne_tropicalmoduli}*{Section 2.6} as well as \cite{Ulirsch_functroplogsch}*{Section 3.5}).

\subsection{Kato fans -- a reminder} Denote the dual category of the category of sharp monoids by $\mathbf{Aff}$. We refer to an object  of this category as an \emph{affine Kato fan}. As proposed in \cite{Kato_toricsing} (also see \cite[Section 3.1]{Ulirsch_functroplogsch}) one can visualize the affine Kato fan corresponding to a sharp monoid $P$ as a sharp monoidal space $\Spec P$, the \emph{spectrum of $P$:} 
\begin{itemize}
\item As a set $\Spec P$ is the set of prime ideals of $P$, i.e. the set of subsets $\frakp$ of $P$ such that $ p+ \frakp\subseteq \frakp$ for all $p\in P$ and the complement $P-\frakp$ is a monoid;
\item the topology on $\Spec P$ is generated by the basic open subsets $D(f)=\{\frakp\vert f\notin\frakp\}$ for $f\in P$; and
\item the structure sheaf is defined by associating to $D(f)$ the sharp localization $\overline{P}_f=P_f/P_f^\ast$ of $P$ along $f$. 
\end{itemize}

\begin{definition}
A \emph{Kato fan} $F$ is a sharp monoidal space that admits a covering $U_i$ by open subsets that are isomorphic to affine Kato fans $\Spec P_i$. 
\end{definition}

A \emph{morphism} of Kato fans is a morphism of sharp monoidal spaces. Throughout this article we assume that every Kato fan is \emph{locally fine and saturated (fs)}, i.e. the monoids $P_i$ in the covering of $F$ above can be chosen to be fine and saturated. We write $\mathbf{Fan}$ for the category of (locally fine and saturated) Kato fans and refer the reader to \cite{GillamMolcho_manifoldswithcorners} and \cite{Gillam_fans} for the state of the art in the theory of Kato fans as well as to \cite[Section 3]{Ulirsch_functroplogsch} for many explicit examples. 

The category $\mathbf{Fan}$ carries a Grothendieck topology, the \emph{big Zariski topology}; its covers are generated by open subsets on $\mathbf{Aff}$ coming from sharp localizations $P\rightarrow \overline{P}_f$ along submonoids generated by an element $f\in P$. 

Every Kato fan defines a presheaf
\begin{equation*}\begin{split}
h_F\mathrel{\mathop: } \mathbf{Fan}^{op}&\longrightarrow \mathbf{Sets}\\
G&\longmapsto F(G)=\Hom(G, F)
\end{split}\end{equation*}
that is a sheaf on $(\mathbf{Fan},\tau_{Zar})$ by the Zariski local definition of $F$ and since every cover of an affine Kato fan already has to contain an isomorphism. By the Yoneda Lemma the association $F\mapsto h_F$ defines a fully faithful functor from the category of Kato fans to the category of sheaves on $\mathbf{Fan}$. In the following we are going to identify a Kato fan $F$ with its functor of points $h_F$ and simply write $F$ for both the locally monoidal space and the functor.

\begin{proposition}\label{prop_fiberproductsKatofans}
The category of locally fine and saturated Kato fans admits fiber products.
\end{proposition}

\begin{proof}
Given a diagram 
\begin{equation*}\begin{CD}
 @.G\\
 @. @VVV\\
F@>>>H
\end{CD}\vspace{0.5em}
\end{equation*}
of Kato fans, we have to construct a Kato fan $F\times_HG$ together with morphism to $F$ and $G$ that makes the above diagram cartesian in the category of Kato fans. If $F=\Spec P$, $G=\Spec Q$, and $H=\Spec S$ are all affine, we can set $F\times_HG=\Spec \overline{P\oplus_SQ}$, where $P\oplus_SQ$ is the amalgamated sum in the category of fine and saturated monoids. This is a fiber product, since for every sharp monoidal space $X$ we have: 
\begin{equation*}\begin{split}
\Hom\big(X,F\times_HG\big)&=\Hom\big(P\oplus_SQ,\calO_X(X)\big)\\
&=\Hom\big(P,\calO_X(X)\big)\times_{\Hom(S,\calO_X(X))}\Hom\big(Q,\calO_X(X)\big) \ .
\end{split}\end{equation*}
The general case now follows by glueing. 
\end{proof}

\begin{definition}
A morphism $f\mathrel{\mathop:}F\rightarrow G$ of locally fine and saturated Kato fans is said to be \emph{strict}, if the induced map $f^{-1}\calO_G\rightarrow \calO_F$  is an isomorphism.
\end{definition}

\begin{proposition}\label{prop_strictmors}
\begin{enumerate}[(i)]
\item A morphism $f\mathrel{\mathop:}F\rightarrow G$ between two Kato fans is strict if and only if it is a local isomorphism. 
\item The class of strict morphisms is stable under composition and base change.
\end{enumerate}
\end{proposition}

\begin{proof}
Let $f\colon F\rightarrow G$ be a morphism of Kato fans. For $x\in F$ denote by $U_x$ the smallest open subset in $F$ containing $x$ and by $V_{f(x)}$ the smallest open subset in $G$ containing $f(x)$. Notice that in this case, we have that $U_x=\Spec P_x$ and $V_{f(x)}=\Spec Q_{f(x)}$ for unique sharp fine and saturated monoids $P_x$ and $Q_{f(x)}$, the local monoids at $x$ and $f(x)$, and that both $x$ and $f(x)$ are the unique closed points of $\Spec P_x$ and $\Spec Q_{f(x)}$ respectively. The morphism $f$ is strict if and only if the induced maps $P_x\rightarrow Q_{f(x)}$ are an isomorphism for all $x\in F$, which is the case if and only if the induced maps $\Spec P_x\rightarrow\Spec Q_{f(x)}$ are isomorphisms. This shows that $f$ is a local isomorphism. Conversely, if $f$ is a local isomorphism, then the induced maps $U_x\rightarrow V_{f(x)}$ are isomorphisms, since both $U_x$ and $V_{f(x)}$ are the smallest open subsets containing $x$ and $f(x)$ respectively. This shows (i), which, in turn, immediately implies (ii). 
\end{proof}

The category $\mathbf{Fan}$ can be endowed with a further Grothendieck topology, the \emph{strict topology}, whose coverings are generated by strict morphisms. By Proposition \ref{prop_strictmors} (ii) the strict topology is, in fact, equivalent to the big Zariski topology on $\mathbf{Fan}$.


\subsection{Kato stacks}
\begin{definition}
A \emph{Kato stack} is a stack $\calF$ over the site $(\mathbf{Fan},\tau_{Zar})$ fulfilling the following two properties:
\begin{enumerate}[(i)]
\item The diagonal morphism $\calF\rightarrow \calF\times\calF$ is representable by Kato fans.
\item There is a morphism $U\rightarrow \calF$ from a Kato fan $U$ to $\calF$ that is strict and surjective. 
\end{enumerate}
\end{definition}

Note hereby that $U\rightarrow \calF$ is representable, since the diagonal of $\calF$ is representable. So $U\rightarrow \calF$ being strict and surjective means that all of its base changes along a morphism to $\calF$ from another Kato fan are strict and surjective. We are going to refer to the strict and surjective morphism $U\rightarrow \calF$ and, in a slight abuse of notation, to $U$ itself as an \emph{atlas} of $\calF$. 

The category of Kato stacks (denoted by $\mathbf{K.St.}$) is the full subcategory of categories fibered in groupoids over $\mathbf{Fan}$, whose objects are Kato stacks. Note hereby that a category fibered in groupoids over $\mathbf{Aff}$ uniquely extends to a stack over $(\mathbf{Fan},\tau_{Zar})$, since every cover of an affine Kato fan always has to contain an isomorphism (see \cite[Proposition 2.3]{CavalieriChanUlirschWise_tropstack} for an analogous statement over the category of rational polyhedral cones). Throughout this article we are again going to assume that all of our Kato stacks are locally fine and saturated, i.e. that there is an atlas $U\rightarrow\calF$ such that $U$ is locally fine and saturated.

By \cite{Ulirsch_functroplogsch}*{Proposition 3.7} the category of Kato fans is equivalent to the category of rational polyhedral cone complexes (in the sense of \cite{KKMSD_toroidal}*{Def. 2.1.5}). As observed in \cite{CavalieriChanUlirschWise_tropstack}*{Remark 5.6}, this equivalence extends to an equivalence of geometric contexts and therefore the $2$-category of \emph{cone stacks} in the sense of \cite{CavalieriChanUlirschWise_tropstack}*{Definition 2.7} is equivalent to the $2$-category of Kato stacks, as defined in this section. So we can immediately translate between the language of Kato stacks and the language of cone stack as developed in \cite{CavalieriChanUlirschWise_tropstack}.

\begin{example}
 Since the presheaf $h_F$ is a sheaf, every Kato fan $F$ defines a Kato stack, which, in a slight abuse of notation, we are simply going to denote by $F$ as well. 
\end{example}

\begin{proposition}\label{proposition_groupoidquot=Katostack}
Let $(R\rightrightarrows U)$ be a strict surjective groupoid object in the category of Kato fans. Then the quotient stack $\big[U\big/R\big]$ is a Kato stack. 
\end{proposition}

The proof of Proposition \ref{proposition_groupoidquot=Katostack} proceeds in complete analogy with the algebraic case (see \cite[Tag 04TJ]{stacks-project}); due to the triviality of descent with respect to the chaotic topology, we may, however, avoid some technicalities.

\begin{proof}[Proof of Proposition \ref{proposition_groupoidquot=Katostack}]
The groupoid $(R\rightrightarrows U)$ defines a functor
\begin{equation*}
\mathbf{Aff}^{op}\longrightarrow \mathbf{Groupoids}
\end{equation*}
by associating to $T\in\mathbf{Aff}$ the groupoid $\big(R(T)\rightrightarrows U(T)\big)$ and the associated category $\big[U\big/R\big]$ fibered in groupoids uniquely extends to a stack over $\mathbf{Fan}$.

In order to show that the diagonal of $\big[U\big/R\big]$ is representable by Kato fans we only have to show that for a Kato fan $T$ and two objects $x,y\in\big[U\big/R\big](T)$ the sheaf $\Isom\big(x\vert_T,y\vert_T\big)$ is representable by a Kato fan. But this follows, since we have a natural cartesian diagram
\begin{equation*}\begin{CD}
\Isom\big(x\vert_T,y\vert_T\big)@>>> R\\
@VVV @VVV\\
T@>(x\vert_T,y\vert_T)>> U\times U \ .
\end{CD}\end{equation*}

The natural morphism $R\rightarrow U\times_{[U/R]} U$ is an equivalence and therefore, for every $x\in \big[U\big/R\big](T)$, there are natural equivalences
\begin{equation*}
U\times_{[U/R]} T\simeq (U\times_{[U/R]} U)\times_{s,U,x} T\simeq R\times_{s,U,x}T \ . 
\end{equation*}
By assumption the projection morphism $R\times_U T\rightarrow T$ is surjective and strict as a base change of $s\mathrel{\mathop:}R\rightarrow U$ and thus the natural map $U\rightarrow \big[U\big/R\big]$ is a strict atlas of $\big[U\big/R\big]$.
\end{proof}

A \emph{presentation} of a Kato stack $\calF$ consists of a groupoid $(R\rightrightarrows U)$ in $\mathbf{Fan}$ together with an equivalence $\big[U\big/R\big]\simeq \calF$. Let $U\rightarrow \calF$ be a strict atlas of a Kato stack $\calF$. The fiber product $U\times_\calF U$ is representable by a Kato fan $R$ and the projections $\pr_0,\pr_1\mathrel{\mathop:}R\rightrightarrows U$, together with the natural composition morphism 
\begin{equation*}
R\times_U R\simeq U\times_\calF U\times_\calF U \xlongrightarrow{\ \ \ \ \pr_{02}\ \ \ } U\times_\calF U\simeq R \ ,
\end{equation*}
define a groupoid in the category of Kato fans. Its quotient stack $\big[U\big/R\big]$ is equivalent to the Kato stack $\calF$ we started with. We refer the reader to \cite[Tag 04T3]{stacks-project} for an analogous construction in the category of schemes.

\subsection{Combinatorial Kato stacks}

Although not immediately apparent, Kato stacks are fundamentally combinatorial objects, very much similar (although not equal) to generalized cone complexes (in the sense of \cite{AbramovichCaporasoPayne_tropicalmoduli}). We finish with a construction (in analogy with  \cite{CavalieriChanUlirschWise_tropstack}*{Section 2.2}) that makes the combinatorial nature of Kato stacks more apparent.  

\begin{definition}
Denote by $\mathbf{Aff}^{st}$ the category of affine Kato fans with strict morphisms as maps. A \emph{combinatorial Kato stack} is a category fibered in groupoids over $\mathbf{Aff}^{st}$.
\end{definition}

One should think of a combinatorial Kato stack as a diagram in $\mathbf{Aff}$ consisting of face morphisms that fulfills the axioms of a category fibered in groupoids. We refer the reader to \cite[Section 2.2]{CavalieriChanUlirschWise_tropstack} for a more careful explanation of this point of view and several illuminating examples.

\begin{proposition}\label{prop_combinatorialKatostacks}
The $2$-category of combinatorial Kato stacks is equivalent to the $2$-category of Kato stacks.
\end{proposition}

Proposition \ref{prop_combinatorialKatostacks} is a translation of \cite{CavalieriChanUlirschWise_tropstack}*{Proposition 2.18} to the realm of Kato fans. Its proof immediately transfers to this situation and is left to the avid reader. Given a Kato stack $\calF$, the associated combinatorial Kato stack $\calF^{comb}$ is given by the category of strict morphisms $\Spec P\rightarrow\calF$ from affine Kato fans into $\calF$. Conversely, we obtain the Kato stack $\calF$ associated to a combinatorial Kato stack $\calF^{comb}$ as the functor
\begin{equation*} 
\Spec P\longmapsto\HOM\big((\Spec P)^{comb},\calF^{comb}\big) 
\end{equation*} 
that sends $\Spec P$ to the groupoid of morphisms $(\Spec P)^{comb}\rightarrow \calF^{comb}$ of categories fibered in groupoids over $\mathbf{Aff}^{st}$. 

\subsection{From Kato stacks to generalized cone complexes}
 We have introduced the theory of Kato stacks to generalize the older (and inherently $1$-categorical) notion of a \emph{generalized (rational polyhedral) cone complex} into a natural $2$-categorical setting. Recall from \cite{AbramovichCaporasoPayne_tropicalmoduli}*{Section 2.6} (also see \cite{Ulirsch_functroplogsch}*{Section 3.5}):
 
 \begin{definition} A \emph{generalized cone complex} $\Sigma$ is a topological space $\vert \Sigma\vert$ that is given as a topological colimit of a diagram of (sharp) rational polyhedral cones $\sigma_\alpha$ with (not necessarily proper) face maps between them.
 \end{definition} 
 
 A \emph{morphism} $f\colon\Sigma\rightarrow \Sigma'$ of generalized cone complexes is given by a continuous map $\vert f\vert \colon \vert \Sigma\vert \rightarrow \vert\Sigma'\vert$ such that for every cone $\sigma\rightarrow \Sigma$ in $\Sigma$ there exists a cone $\sigma'\rightarrow \Sigma'$ in $\Sigma'$ such that $\vert f\vert$ factors through a $\Z$-linear map $\sigma\rightarrow\sigma'$ of rational polyhedral cones.

Let $P$ be a fine and saturated monoid. We write $\sigma_P$ for the rational polyhedral cone $\Hom(P,\R_{\geq 0})$ associated to $P$. The association $P\mapsto \sigma_P$ defines a natural functor from the category of fine and saturated monoids to the category of rational polyhedral cones.   

\begin{proposition}
There is a natural functor 
\begin{equation*}\begin{split}
\Sigma_{(.)}\colon\mathbf{K.St.}&\longrightarrow \mathbf{GCC}\\
\calF&\longmapsto \Sigma_\calF
\end{split}\end{equation*}
from the $2$-category of Kato stacks to the $1$-category of generalized cone complexes that sends an affine Kato fan $\Spec P$ to the cone $\sigma_P$ and preserves colimits.
\end{proposition}

\begin{proof} Given a Kato stack $\calF$, the cones defining $\Sigma_\calF$ are exactly the cones $\sigma_\xi=\sigma_P$ associated to strict morphisms $\xi\colon \Spec P\rightarrow \calF$ with face maps over $\Sigma_\calF$ induced by strict morphisms over $\calF$. In other words, the generalized cone complex $\Sigma_\calF$ is obtained by applying the functor $\sigma_{(.)}$ to the the combinatorial Kato stack $\calF^{comb}$, thought of as a diagram in $\mathbf{Aff}^{st}$, and taking colimits in the category of topological spaces. The association $\calF\mapsto\Sigma_\calF$ is clearly functorial. 
\end{proof}


\section{Artin fans}\label{section_Artinfans}

Let $P$ be a monoid and write $U_P$ for the affine toric variety $\Spec k[P]$ with big torus $T=\Spec k\big[P^{gp}\big]$. We refer to the quotient stack $\calA_P=\big[U_P\big/T\big]$ as an \emph{Artin cone} (see Figure \ref{figure_Artincone} for a visualization). Recall the following Definition \ref{definition_Artinfans} from \cite{AbramovichChenMarcusWise_boundedness}*{Section 3.1}.

\begin{figure}[h]\begin{tikzpicture}
\fill (1,1) circle (0.20 cm)
      (3,1) circle (0.12 cm)
      (1,3) circle (0.12 cm)
      (3,3) circle (0.05 cm);
\draw [->] (1.5,1) -- (2.6,1);
\draw [->] (1.4,3) -- (2.7,3);
\draw [->] (1,1.5) -- (1,2.6);
\draw [->] (3,1.4) -- (3,2.7);

\node at (0.5,0.5) {$1$};
\node at (3.5,0.5) {$\G_,$};
\node at (0.5,3.5) {$\G_m$};
\node at (3.5,3.5) {$\G_m^2$};
\end{tikzpicture}\caption{The Artin cone $\calA_{\N^2}$. The arrows indicate specialization and the labels indicate the automorphism groups of each point.}\label{figure_Artincone}\end{figure}
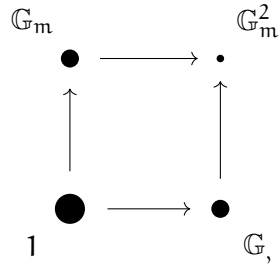

\begin{definition}\label{definition_Artinfans}
An \emph{Artin fan} is a logarithmic algebraic stack that admits a cover by a disjoint union of Artin cones that is representable and strict \'etale. 
\end{definition}

As explained in \cite[Section 5]{Olsson_loggeometryalgstacks} an Artin cone $\calA_P=\big[U_P\big/T\big]$ naturally carries a logarithmic structure making the morphism $\calA_P\rightarrow\LOG_k$ representable and \'etale. Therefore, for every Artin fan $\calA$, the tautological morphism $\calA\rightarrow \LOG_k$ is \'etale and so $\calA$ itself is logarithmically \'etale over $k$\footnote{Earlier versions of both  \cite{AbramovichChenMarcusWise_boundedness} and this article used to define Artin fans a logarithmic algebraic stack that are logarithmically \'etale over $k$. We follow \cite{AbramovichChenMarcusWise_boundedness} and restrict the class of Artin fans for the sake of Theorem \ref{thm_Artinfan=Katostack} (also see Example \ref{example_mu_k} below).}. If $\calA\rightarrow \LOG_k$ is also representable, we say that $\calA$ has \emph{faithful monodromy}. The category of Artin fans (with faithful monodromy) is the full $2$-subcategory of the $2$-category of logarithmic algebraic stacks whose objects are Artin fans (with faithful monodromy respectively). 

We now consider Artin fans without monodromy.

\begin{definition}[Artin fans without monodromy]\label{definition_Artinfanswomonodromy}
A homomorphism $P\rightarrow Q$ of monoids induces a torus-invariant morphism $U_Q\rightarrow U_P$ and therefore a logarithmic morphism $\calA_Q\rightarrow \calA_P$ on the level of quotient stacks. If $P\rightarrow Q$ is a face map, then $\calA_Q\rightarrow\calA_P$ is a representable strict open immersion. Therefore we can associate to every Kato fan $F$ an Artin fan (with faithful monodromy)
\begin{equation*}
\calA_F=\lim_{\longrightarrow}\calA_{P}
\end{equation*}
by glueing the $\calA_P$ over all open affine subsets $U=\Spec P$ of $F$, such that there is a natural isomorphism $\big(\vert\calA_F\vert,\Mbar_{\calA_F}\big)\simeq F$ of sharp monoidal spaces (see Figure \ref{figure_Artinfan}). Artin fans of the form $\calA_F$ are said to be \emph{without monodromy}.
\end{definition}

\begin{figure}[h]\begin{tikzpicture}
\fill (0,0) circle (0.20 cm)
      (2,0) circle (0.12 cm)
      (-2,0) circle (0.12 cm)
      (0,2) circle (0.12 cm)
      (0,-2) circle (0.12 cm)
      (2,2) circle (0.05 cm)
      (2,-2) circle (0.05 cm)
      (-2,2) circle (0.05 cm)
      (-2,-2) circle (0.05 cm);
\draw [->] (0.5,0) -- (1.6,0);
\draw [->] (0.4,2) -- (1.7,2);
\draw [->] (0,0.5) -- (0,1.6);
\draw [->] (2,0.4) -- (2,1.7);
\draw [->] (-0.5,0) -- (-1.6,0);
\draw [->] (-0.4,-2) -- (-1.7,-2);
\draw [->] (0,-0.5) -- (0,-1.6);
\draw [->] (-2,-0.4) -- (-2,-1.7);
\draw [->] (-2,0.4) -- (-2,1.7);
\draw [->] (2,-0.4) -- (2,-1.7);
\draw [->] (0.4,-2) -- (1.7,-2);
\draw [->] (-0.4,2) -- (-1.7,2);

\node at (0.5,0.5) {$1$};
\node at (2.5,0) {$\G_m$};
\node at (0,2.5) {$\G_m$};
\node at (2.5,2.5) {$\G_m^2$};
\node at (-2.5,0) {$\G_m$};
\node at (0,-2.5) {$\G_m$};
\node at (-2.5,-2.5) {$\G_m^2$};
\node at (2.5,-2.5) {$\G_m^2$};
\node at (-2.5,2.5) {$\G_m^2$};
\end{tikzpicture}\caption{The Artin fan $\big[\PP^1\times\PP^1\big/\G_m^2\big]$. It is glued from four copies of the Artin cone $\calA_{\N^2}$ along the open subcones $\calA_{\N}$ as in Definition \ref{definition_Artinfanswomonodromy}.}\label{figure_Artinfan}\end{figure}

\begin{proposition}\label{lemma_Artinfanswomonodromy} The association $F\mapsto \calA_F$ defines a fully faithful functor from the category of Kato fans into the category of Artin fans. Moreover, if $f\mathrel{\mathop:}F_1\rightarrow F_2$ is a strict morphism of Kato fans, the induced morphism $\calA(f)\mathrel{\mathop:}\calA_{F_1}\rightarrow\calA_{F_2}$ is a representable strict \'etale morphism of logarithmic stacks  
\end{proposition}

\begin{proof}
By \cite[Proposition 5.17]{Olsson_loggeometryalgstacks} we have
\begin{equation*}
\Hom_{\mathbf{Log.St.}}(U_Q,\calA_P)=\Hom_{\mathbf{Mon}}\big(P,\Gamma(U_Q,\Mbar_{U_Q})\big)=\Hom_{\mathbf{Mon}}(P,Q)
\end{equation*}
and the left hand side is equal to $\Hom_{\mathbf{Log.St.}}(\calA_Q,\calA_P)$, since every logarithmic morphism $U_Q\rightarrow \calA_P$ is $T_Q=\Spec k[Q^{gp}]$-invariant. This shows that for two sharp monoids $P$ and $Q$ the natural map 
\begin{equation*}
\Hom_{\mathbf{Mon}}(P,Q)\longrightarrow\Hom_{\mathbf{Log.St.}}(\calA_Q,\calA_P)
\end{equation*}
is a bijection.

Let $F$ be a Kato fan. Glueing over open immersions together with the above observation yields the functoriality of the association $F\mapsto\calA_F$. Suppose that $f\mathrel{\mathop:}F_1\rightarrow F_2$ is a strict morphism of Kato fans. Then the induced morphism $\calA(f)\mathrel{\mathop:}\calA_{F_1}\rightarrow\calA_{F_2}$ is Zariski-locally an isomorphism, i.e.,  given an open affine subset $U=\Spec P$ such that $f\vert_U$ induces an isomorphism onto its image, the morphism $\calA(f)$ induces an isomorphism $\calA_U\xrightarrow{\sim}\calA_{f(U)}$. Therefore $\calA(f)$ is necessarily representable and strict \'etale. 
\end{proof}

The following Theorem \ref{thm_Artinfan=Katostack} is a translation of \cite{CavalieriChanUlirschWise_tropstack}*{Theorem 3} to our setting. It shows that Artin fans are lifts of Kato stacks to the category of logarithmic algebraic stacks.

\begin{theorem}\label{thm_Artinfan=Katostack}
There is a natural equivalence between the $2$-category of Artin fans and the $2$-category of Kato stacks 
such that, whenever $\big[U\big/R\big]\simeq \calF$ is a strict groupoid presentation of $\calF$ by Kato fans, the corresponding Artin fan $\calA_\calF$ is the quotient of the strict \'etale groupoid $\big(\calA_R\rightrightarrows \calA_U)$.
\end{theorem}

Note that by Proposition \ref{lemma_Artinfanswomonodromy} the category of Artin fans without monodromy is equivalent to the category of Kato fans. Therefore it is, in particular, only a $1$-category and the quotient of the strict \'etale groupoid $\big(\calA_R\rightrightarrows \calA_U)$ is representable in the $2$-category of logarithmic stacks. 

Using the equivalence between Kato stacks and cone stacks explained above, one can immediately deduce Theorem \ref{thm_Artinfan=Katostack}  from the arguments in \cite{CavalieriChanUlirschWise_tropstack}*{Section 6.4 and 6.5}. Given an Artin fan $\calA$, we obtain a Kato stack $\calF$ such that $\calA\simeq\calA_\calF$ by considering the association
\begin{equation*}\begin{split}
\mathbf{Aff}^{op}&\longrightarrow \mathbf{Groupoids}\\
\Spec P &\longmapsto \HOM_{\mathbf{Log.St.}}(\calA_P,\calA) 
\end{split}\end{equation*}
that sends $\Spec P$ to the groupoid of morphisms $\calA_P\rightarrow \calA$ of logarithmic stacks. 

Conversely, given a Kato stack $\calF$, the associated Artin fan $\calA_\calF$ is the stackification of the fibered category whose fiber over a logarithmic scheme $S$ is the groupoid $\calF\big(\Gamma(S,\Mbar_S)\big)$. In other words, the fiber of $\calA_\calF$ over a logarithmic scheme is the groupoid of collections of objects in $\calF\big(\Mbar_{S,s}\big)$ (indexed by the geometric points of $s$) that are naturally compatible with respect to \'etale specializations. 

\begin{example}
Let $G$ be a (discrete) group. We may think of $G$ as a group object in the category of Kato fans, whose underlying Kato fan is a disjoint union of points $\Spec 0$. Now suppose that $G$ acts on a Kato fan $F$ by automorphisms. The Artin fan $\calA_\calF$ associated to the quotient stack $\calF=\big[F/G\big]$ is the quotient $\big[\calA_F\big/G\big]$. In particular, the Artin fan associated to the classifying stack $\mathbf{B}G$ (over $\mathbf{Aff}$) is nothing but $\mathbf{B}G$ (over $\mathbf{LSch}$ or $\mathbf{Sch}$). 
\end{example}

\begin{example}
We may consider the category $\mathbf{Aff}^{st}$ itself as a category fibered in groupoids over $\mathbf{Aff}^{st}$, i.e. as a combinatorial Kato stack. Denote the associated Kato stack by $\FAN$. We may think of $\FAN$ as the $2$-colimit of the diagram of all affine Kato fans with all possible strict morphisms between them (including all automorphisms). By \cite[Corollary 5.25]{Olsson_loggeometryalgstacks} the $2$-colimit of the induced diagram in the category of the Artin fans is equal to $\LOG_k$, the classifying stack of fine and saturated logarithmic structures over $k$. Therefore $\calA_{\FAN}$ is nothing but $\LOG_k$. 
\end{example}

The following Example \ref{example_mu_k} shows that not every logarithmic algebraic stacks that is logarithmically \'etale over $k$ comes from a Kato stack. We thank Jonathan Wise for allowing us to include it here. 

\begin{example}\label{example_mu_k}
Let $l\geq 1$ and consider the operation 
\begin{equation*}\begin{split}
\mu_l\mathrel{\mathop:}\G_m\times\A^1&\longrightarrow \A^1\\
(t,x)&\longmapsto t^l\cdot x \ .
\end{split}\end{equation*}
The argument in \cite[Section 5]{Olsson_loggeometryalgstacks} shows that the quotient stack $\big[\A^1\big/_{\mu_l}\G_m\big]$ under the operation $\mu_l$ is logarithmically \'etale over $k$. For $l>1$ none of the $\big[\A^1\big/_{\mu_l}\G_m\big]$ admit a strict \'etale cover by a disjoint union of Artin cones, since $\big[\A^1\big/_{\mu_l}\G_m\big]$ is naturally a $\mu_l$-gerbe over $\big[\A^1\big/\G_m]$ which does not admit strict \'etale covers by $\big[\A^1\big/\G_m\big]$.
\end{example}

\begin{remark}
The heuristic in the above Example \ref{example_mu_k} is that Kato stacks (as defined in this article) lack enough structure to encode $\mu_l$-gerbes over $[\A^1/G_m]$. It would be very interesting to investigate whether the various forms of stacky fans, which arise as combinatorial models in the theory of toric stacks \cite{BorisovChenSmith_toricDMstacks, FantechiMannNironi_smoothtoricDMstacks, GeraschenkoSatriano_toricstacksI, GillamMolcho_stackyfans}, can be used to enhance the category of Kato stacks so that examples such as the above are included.
\end{remark}


\section{From logarithmic stacks to Artin fans}\label{section_logstacktoArtinfan} 

In this section we recall from \cite{AbramovichChenMarcusWise_boundedness} how to canonically associate to every logarithmic algebraic stack $\calX$ an Artin fan $\calA_\calX$ (with faithful monodromy) together with a natural strict morphism $\phi_\calX\colon\calX\rightarrow\calA_{\calX}$ and study its basic properties. The goal is to set the stage for the proofs of Theorem \ref{thm_trop=anal} (in Section \ref{section_trop=anal}) and Theorem \ref{thm_ACPArtin} (in Section \ref{section_Artinfansmoduli}) from the introduction. 

\subsection{Artin fans with faithful monodromy} Before we start, recall that an Artin fan $\calA$ is said to have \emph{faithful monodromy} if the tautological strict morphism $\calA\rightarrow \LOG_k$ is representable. 

\begin{proposition}[\cite{AbramovichChenMarcusWise_boundedness}*{Proposition 3.2.1}]\label{prop_logstacktoArtinfan}
Let $\calX$ be a fine and saturated logarithmic stack locally of finite type over $k$. Then there is an Artin fan $\calA_\calX$ with faithful monodromy together with a natural strict morphism $\phi_\calX\colon \calX\rightarrow \calA_\calX$ that is initial among all strict morphisms to an Artin fan with faithful monodromy. 
\end{proposition}

In other words, given a strict morphism $\phi\colon\calX\rightarrow \calA$ to an Artin fan $\calA$ with faithful monodromy there is unique (automatically strict) morphism $\calA_\calX\rightarrow \calA$ that makes the diagram 
\begin{center}\begin{tikzcd}
\calX  \arrow[r, "\phi_\calX"] \arrow[rrd, bend right, "\phi"']& \calA_\calX \arrow[rd,dashed] & \\
& & \calA
\end{tikzcd}\end{center}
commute. The Artin fan $\calX$ associated to a logarithmic stack $\calX$ always has faithful monodromy. So, in particular, if $\calX$ is an Artin fan, then the associated Artin fan $\calA_\calX$ is equal to $\calX$ itself if and only if $\calX$ has faithful monodromy (also see Example \ref{example_M11} below).  

\begin{definition}\label{definition_small} 
A logarithmic scheme $X$ is said to be \emph{small}, if the Artin fan $\calA_X$ of $X$ is an Artin cone $\calA_{P_X}$ for the monoid $P_X=\Gamma(X,\Mbar_X)$ and the induced isomorphism $P_X\xrightarrow{\sim}\Gamma(X,\Mbar_X)$ lifts to a chart of $X$. 
\end{definition}

The proof of Proposition \ref{prop_logstacktoArtinfan} proceeds by first noticing that the property of admitting a factorization as above is stable under taking colimits in the category of representable strict and smooth morphisms $\mathcal{U}\rightarrow \calX$. So, the central point of this proof is to show that every geometric point of $\calX$ has a smooth strict neighborhood that is small. 

An immediate application of Proposition \ref{prop_logstacktoArtinfan} shows that the construction of $\calA_\calX$ is functorial with respect to strict morphisms. 

\begin{corollary}
Let $f\colon\calX\rightarrow\calY$ be a strict morphism of fine and saturated logarithmic algebraic stacks locally of finite type over $k$. Then there is a natural strict morphism $\calA(f)\colon\calA_\calX\rightarrow\calA_\calY$ that makes the diagram
\begin{center}\begin{tikzcd}
\calX \arrow[r,"\phi_\calX"] \arrow[d,"f"']& \calA_\calX\arrow[d, "\calA(f)"]\\
\calY \arrow[r,"\phi_\calY"] & \calA_\calY
\end{tikzcd}\end{center}
commute.
\end{corollary}

In general, the construction of $\phi_\calX\colon \calX\rightarrow \calA_\calX$ fails to be functorial with respect to a general logarithmic morphism. We refer the reader to \cite[Example 3.3.1]{AbramovichChenMarcusWise_boundedness} and \cite[Section 5.4.1]{AbramovichChenMarcusUlirschWise_logsurvey} for a counterexample. This failure of functoriality is due to the requirement that $\calA_\calX$ has faithful monodromy, which, on the one hand, makes the automorphism groups present in $\calA_{\calX}$ completely combinatorial but, on the other hand, forgets some crucial information about the monodromy of $\calX$ that operates trivially on the characteristic monoid $\Mbar_\calX$. 

\begin{remark}
The notion of a \emph{small} logarithmic scheme has already been defined in \cite{Ulirsch_functroplogsch}*{Section 4.2} and in \cite{AbramovichChenMarcusUlirschWise_logsurvey}*{Section 5.3}. There, however, we did not require the isomorphism $P_X\xrightarrow{\sim}\Gamma(X,\Mbar_X)$ to lift to a chart of $X$ and so the condition in Definition \ref{definition_small} is stronger than the previous one. In this article we prefer this stronger condition to the earlier one, since existence of a chart allows us to choose a (usually non-unique) strict morphism to the toric variety $\Spec k[P_X]$. Our new definition is closer in spirit to the analogous definition of a small toric chart in \cite{AbramovichCaporasoPayne_tropicalmoduli}*{Definition 6.2.4}.
\end{remark}


\subsection{Logarithmic stratifications and the Artin fan}
Suppose now that $\calX$ is a logarithmically smooth stack. There is a natural stratification of $\calX$ into locally closed connected substacks that is \'etale locally in toric charts induced from the stratification of a toric variety into torus-orbits. 

Denote by $\calX_0$ the open substack on which the characteristic monoid $\Mbar_\calX=M_\calX/M_\calX^\ast$ is trivial. Then $\calX_0$ is smooth and its connected components are the strata of codimension $0$ in $\calX$. For $n\geq 1$ we inductively define $\calX_{n}$ to be the open substack of regular points of $\calX-(\calX_0\cup \cdots \cup\calX_{n-1})$; the connected components of $\calX_n$ are the strata of $\calX$ of codimsion $n$. 

The strata of $\calX$ are parametrized by the Zariski points $\big\vert\calA_\calX\big\vert$ of $\calA_\calX$: Given a point $x$ in $\big\vert \calA_\calX\big\vert$, the underlying set in $\big\vert\calX\big\vert$ in the associated stratum is nothing but the preimage 
\begin{equation*}
\phi_\calX^{-1}(x)\subseteq \big\vert\calX\big\vert \ .
\end{equation*}

Define a combinatorial cone stack $\calF_{\calX}^{comb}$ whose objects are the generic points $\eta_E$ of the strata $E$ of $\calX$ and whose morphisms are generated by the following two classes of arrows:
\begin{itemize}
\item the \'etale specializations $\eta_E\rightsquigarrow\eta_{E'}$, whenever $E'$ is a stratum in the closure of a stratum $E$ but not equal to $E$; and 
\item the image of the natural monodromy representation $\pi_1(E,y_E)\rightarrow \Aut(\Mbar_{X,\eta_E})$ of the fundamental group $\pi_1(E,y_E)$ of $E$ for some geometric point $y_E$ of $E$.
\end{itemize}
For the latter we refer the reader to \cite{Noohi_stackyfundamentalgroups} for the basic theory of \'etale fundamental groups of algebraic stacks; notice in particular that the above monodromy representation does not depend on the choice of $y_E$, since every stratum is connected. 

\begin{proposition}\label{prop_Artinfanoflogsmoothstack}
Let $\calX$ be a logarithmically smooth stack and denote by $\calF_\calX$ the Kato stack associated to the combinatorial Kato stack $\calF_{\calX}^{comb}$. Then there is a natural isomorphism $\calA_{\calX}\simeq\calA_{\calF_\calX}$.
\end{proposition}

\begin{proof}
It is enough to show that there is a natural equivalence between the combinatorial Kato stack associated to $\calA_{\calX}$ and $\calF_\calX^{comb}$ defined above. This follows from the following observations:

Let $U\rightarrow \calX$ be a smooth strict morphism from a small logarithmic scheme. Then $U$ is also logarithmically smooth and has a unique closed stratum 
\begin{equation*}
\big\{x\in U\big\vert\Gamma(U,\Mbar_U)\xrightarrow{\sim}\Mbar_{U,x} \textrm{ is an isomorphism.}\big\} \ .
\end{equation*} 
Moreover, there is a unique stratum $E$ of $\calX$ such that the preimage of $E$ in $U$ is the unique closed stratum of $U$. 

Now let $U'\rightarrow U$ be a strict smooth morphisms from a small logarithmic scheme. Then the restriction $\Mbar_\calX(U\rightarrow \calX)\rightarrow\Mbar_\calX(U'\rightarrow \calX)$ induces a strict \'etale morphism $\calA_{U'}\rightarrow \calA_{U}$. If the map $\calA_{U'}\rightarrow \calA_{U}$ is not an isomorphism, it determines an \'etale specialization $\eta_{E'}\rightsquigarrow\eta_E$ on the generic points of the strata $E'$ and $E$ (corresponding to the closed strata of $U$ and $U'$ respectively). 

Finally, the automorphisms of an object in $\HOM(\calA_P,\calA_\calX)$ are a subgroup of the monoid automorphisms $\Aut(P)$ of $P$, since $\calA_\calX\rightarrow\LOG_k$ is representable. This subgroup is precisely the image of the monodromy representation, since every automorphism of a representable \'etale cover of a stratum $E$ is realized by an automorphism of a suitable \'etale cover of a strict smooth neighborhood $U\rightarrow\calX$ of $\eta_E$ where $U$ is small and the preimage of $E$ is the closed stratum of $U$. 
\end{proof}

Let $\calF$ be a Kato stack and denote $\calF_{wom}$ the Kato stack whose combinatorial Kato stack is given as the following category fibered in groupoids over $\mathbf{Aff}^{st}$:
\begin{itemize}
\item The objects in the fiber over $\Spec P\in \mathbf{Aff}^{st}$ are precisely the same objects as the ones in $\calF(\Spec P)$; and 
\item the automorphisms of an object $\xi\in\calF(\Spec P)$ are precisely the image of the natural homomorphism $\Aut_\calF(\xi)\rightarrow \Aut_{\mathbf{Mon}}(P)$.
\end{itemize}

\begin{corollary}\label{cor_ArtinfanofArtinfan}
Let $\calF$ be a Kato stack. 
\begin{enumerate}[(i)]
\item The fibered category $\calF_{wom}$ is a Kato stack without monodromy and the natural morphism $\calA_\calF\rightarrow \calA_{\calF_{wom}}$ is initial among all strict morphisms to Artin fans without monodromy.  
\item The induced map $\Sigma_{\calF}\rightarrow\Sigma_{\calF_{wom}}$ is an isomorphism of generalized cone complexes. 
\end{enumerate}
\end{corollary}

In other words, the Artin fan associated to $\calA_\calF$ in the sense of Proposition \ref{prop_logstacktoArtinfan} is given by $\calA_{\calF_{wom}}$.

\begin{proof}[Proof of Corollary \ref{cor_ArtinfanofArtinfan}]
An immediate application of Proposition \ref{prop_Artinfanoflogsmoothstack}, applied to the Artin fan $\calA_\calF$, yields part (i). Consider Part (ii): We may think of both $\calF$ and $\calF_{woz}$ as diagrams in the category $\mathbf{Aff}^{st}$, which are equal, once we ignore the $2$-categorical extra structure of $\calF$. Thus, after applying the functor $\sigma_{(.)}$, the $1$-categorical colimits of the diagrams $\calF^{comb}$ and $\calF_{woz}^{comb}$ are equal. 
\end{proof} 

\begin{example}[Toric varieties]\label{example_toricArtinfan}
Let $P$ be a monoid. The Artin fan of the affine toric variety $U_P=\Spec k[P]$ is the Artin cone $\calA_P=\big[U_P\big/T\big]$ and the tautological morphism is the quotient map $U_P\rightarrow \calA_P$. For a general $T$-toric variety $Z$, glueing over $T$-invariant open affine subsets yields that the associated Artin fan $\calA_Z$ is the quotient stack $\big[Z\big/T\big]$ and the tautological morphism is the quotient map $Z\rightarrow \big[Z\big/T\big]$.
\end{example}

\begin{example}[Logarithmic schemes without monodromy]\label{example_logschemeswithoutmonodromy}
Recall now from \cite[Section 4.3]{Ulirsch_functroplogsch} that a Zariski logarithmic scheme $X$ is said to have \emph{no monodromy} if there is a strict morphism $(X,\Mbar_X)\rightarrow F$ of sharp monoidal spaces into a Kato fan $F$. In this case, by \cite[Proposition 4.14]{Ulirsch_functroplogsch}, there is a strict characteristic morphism $\phibar_X\mathrel{\mathop:}(X,\Mbar_X)\rightarrow F_X$ into a Kato fan $F_X$ that is initial among all strict morphisms to Kato fans. 

The Artin fan $\calA_X$ constructed in Proposition \ref{prop_logstacktoArtinfan} above is naturally isomorphic to $\calA_{F_X}$ and the induced diagram
\begin{center}\begin{tikzcd}
 (X,\Mbar_X) \arrow[rd, "\vert\phi_X\vert"'] \arrow[rrd,bend left, "\phibar_X"]&&\\
& \big(\vert\calA_X\vert,\Mbar_{\calA_X}\big) \arrow[r,"\sim"] & F_X
\end{tikzcd}\end{center}
commutes.
\end{example}


\subsection{Artin fans and moduli spaces}\label{section_Artinfansmoduli}

We now conclude this section with the proof of Theorem \ref{thm_ACPArtin} from the introduction. We remind the reader of the following definition from \cite{CavalieriChanUlirschWise_tropstack}.

\begin{definition}
Let $P$ be a monoid. A \emph{tropical curve with edge lenghts in $P$} is a tuple $\Gamma = (V,E,L,h,m,d)$ consisting of 
\begin{itemize}
\item a finite graph $(V,E,L)$ with $n=\# L$ legs, 
\item a vertex weight $h\colon V\rightarrow \Z_{\geq 0}$, as well as
\item a marking of the legs, i.e. a bijection $m\colon\{1,\ldots, n\}\xrightarrow{\sim}L$, and
\item an edge length $d\colon E\rightarrow P-\{0\}$. 
\end{itemize}
\end{definition}

The \emph{genus} of $\Gamma$ is defined to be
\begin{equation*}
g(\Gamma)=b_1(G)+\sum_{v\in V}h(v)
\end{equation*}
and we say that $\Gamma$ is \emph{stable} if for every vertex $v\in V$ we have 
\begin{equation*}
2h(v)-2+\vert v\vert >0 \ , 
\end{equation*}
where $\vert v\vert$ denotes the valence of $v$ in the graph $G$.

Let $2g-2+n>0$. The \emph{moduli stack $\calM_{g,n}^{trop}$ of stable $n$-marked tropical curves of genus $g$} is the unique stack over $(\mathbf{Fan},\tau_{Zar})$ whose fiber over an affine Kato fan $\Spec P$ is the groupoid of tropical curves of genus $g$ and $n$ marked points, whose edge lengths lie in $P$. Using the identification between Kato stacks and cone stacks, one may rephrase \cite[Theorem 1]{CavalieriChanUlirschWise_tropstack} as saying that $\calM_{g,n}^{trop}$ is representable by a Kato stack.

In \cite[Section 3.4]{CavalieriChanUlirschWise_tropstack}, expanding on \cite{ChanGalatiusPayne_tropicalmoduliII, Ulirsch_tropHassett}, one can find an explicit description of $\calM_{g,n}^{trop}$ as a combinatorial Kato stack $J_{g,n}$, as follows: The objects of $J_{g,n}$ are finite stable vertex-weighted graphs $G$ of genus $g$ with $n$ marked legs and whose morphisms are generated by weighted edge contractions and automorphisms of $G$. A \emph{weighted edge contraction} is a graph contraction $\pi\colon G\rightarrow G'$ such that for every vertex $v\in V(G')$ we have $g(\pi^{-1}(v))=h(v)$, and an \emph{automorphism} of $G$ is a graph automorphism that respects both the vertex weight $h$ and the marking of the legs. The natural association $G\mapsto \Spec \N^{E(G)}$ makes $J_{g,n}$ into a combinatorial Kato stack, i.e. a category fibered in groupoids over $\mathbf{Aff}^{st}$: Let $G$ be an object of $J_{g,n}$. An affine open subset $U$ of $\Spec \N^{E(G)}$ is of the form $U=\Spec \N^{E'}$ for a subset $E'\subseteq E(G)$ and the restriction of $G\vert_U$ is given by contracting exactly the edges in $E(G)-E'$.

\begin{proof}[Proof of Theorem \ref{thm_ACPArtin}]
It is a well-known fact that the objects of $J_{g,n}$ parametrize the boundary strata of $\calMbar_{g,n}$ and that a  locally closed stratum $\calM_G$ (parametrizing stable curves with dual graph $G$) lies in the closure of another locally closed stratum $\calM_{G'}$ if and only if there is a weighted edge contraction $G'\rightarrow G$ (see \cite{AbramovichCaporasoPayne_tropicalmoduli}*{Section 3.3 and 7.1}). 

As explained in \cite[Proposition 7.2.1]{AbramovichCaporasoPayne_tropicalmoduli}, the structure results for the strata $\calM_G$ in $\calMbar_{g,n}$ in \cite{ArbarelloCornalbaGriffiths_moduliofcurves}*{Chapter XII Proposition (10.11)} imply that the image of the monodromy representation $\pi_1(\calM_G,y_G)\rightarrow \Aut(\N^{E(G)})$ is nothing but the image of the natural homomorphism $\Aut(G)\rightarrow\Aut(\N^{E(G)})$. This shows that, by Proposition \ref{prop_Artinfanoflogsmoothstack} and Corollary \ref{cor_ArtinfanofArtinfan}, the Artin fan $\calA_{\calM_{g,n}^{log}}$ is naturally isomorphic to the Artin fan of $\calM_{g,n}^{trop}$. 

The generalized cone complex of $\calM_{g,n}^{trop}$ is naturally isomorphic to $M_{g,n}^{trop}$ by \cite[Section 4.3]{AbramovichCaporasoPayne_tropicalmoduli} (also see \cite{ChanGalatiusPayne_tropicalmoduliII, Ulirsch_tropHassett}). Therefore Theorem \ref{thm_trop=anal} as well as Corollary \ref{cor_ArtinfanofArtinfan} (ii) imply that the induced map $M_{g,n}^{trop}\rightarrow \Sigma_{\calM_{g,n}^{log}}$ is an isomorphism of generalized cone complexes. 

Since, by \cite[Theorem 4]{CavalieriChanUlirschWise_tropstack} the natural tropicalization map $\trop_{g,n}\colon \calM_{g,n}^{log}\rightarrow \calA_{\calM_{g,n}^{trop}}$ is strict, the natural diagram 
\begin{center}\begin{tikzcd}
\calM_{g,n}^{log}\arrow[rrd,bend left, "\phi_{g,n}"] \arrow[rd,"\trop_{g,n}"']&&\\
& \calA_{\calM_{g,n}^{trop}} \arrow[r] & \calA_{\calM_{g,n}^{log}}
\end{tikzcd}\end{center}
commutes. So, applying the functor $(.)^\beth$, we immediately obtain the commutativity of the diagram in Theorem \ref{thm_ACPArtin}. 
\end{proof}

\begin{example}[The moduli stack $\calM_{1,1}^{trop}$]\label{example_M11}
Let $\Gamma$ be a trivalent tropical curve of genus one with one marked leg with underlying graph $G$. Then $\Gamma$ naturally has a non-trivial automorphism given by flipping the loop. Its image in the automorphism group of $\sigma_G=\R_{\geq 0}$ is trivial and therefore the Artin fan of $\calA_{\calM_{1,1}^{trop}}$ is not equal to $\calA_{\calM_{1,1}^{trop}}$ itself, but rather to $\calA_{\calM_{1,1}^{log}}=\big[\A^1/\G_m\big]$ (see Figure \ref{figure_M11}).
\end{example}

\begin{figure}[h]\begin{tikzpicture}
\fill (0,0) circle (0.12 cm);
\fill (2,0) circle (0.05cm);
\draw [->] (0.3,0) -- (1.8,0);

\node at (0,0.5) {$0$};
\node at (2,0.5) {$\Z_2$};
\node at (-1,0) {$\calM_{1,1}^{trop}$};

\fill (5,0) circle (0.12cm);
\fill (7,0) circle (0.05cm);
\draw [->] (5.3,0) -- (6.8,0);
\node at (5,0.5) {$0$};
\node at (7,0.5) {$0$};
\node at (9,0) {$\calF_{\calM_{1,1}^{log}}=\Spec \N$};

\end{tikzpicture}\caption{The Kato stacks $\calM_{1,1}^{trop}$ and $\calF_{\calM_{1,1}^{log}}=\Spec\N$. The labels indicate automorphism groups at each point.} \label{figure_M11}\end{figure}


\section{Stacky generic fibers}\label{section_stackygenericfiber}

\subsection{Thuillier's generic fiber functor} Let $k$ be a non-Archimedean field. In \cite{Berkovich_book, Berkovich_etalecoho} Berkovich has introduced a refinement of the classical rigid analytic spaces in the sense of Tate \cite{Tate_rigidanalyticspaces} with favorable topological properties. We refer the reader to \cite{Berkovich_etalecoho} and \cite{Temkin_Berkovichspaces} for background and the standard notations in this theory. 

Denote the absolute value on $k$ by $\vert.\vert$ and write $\vert k^\ast\vert$ for the value group of $\vert .\vert$, a subgroup of $(\R_{> 0},\cdot)$. We write $\calM(\calB)$ for the Berkovich spectrum of a $k$-affinoid algebra, as well as $\calH(x)$ for the complete residue fields at a point $x$ of a $k$-analytic space $X$ and $\calH(x)^\circ$ for its valuation ring. If $X$ is a scheme that is locally of finite type over $k$, then $X^{an}$ denotes the analytification of $X$. Moreover, recall that an affinoid algebra $\calB$ is said to be \emph{strict} if there is an admissible surjecitve homomorphism $k\{r_1^{-1}T_1, \cdots, r_n^{-1}T_n\}\twoheadrightarrow \calB$ from a Tate algebra with $r_i\in \vert k^\ast\vert$ and that an analytic space is called \emph{strict} if it can be covered by strict affinoid domains. 

 Suppose that $k$ is carrying the trivial absolute value. Let $X$ be a scheme that is locally of finite type over $k$. Consider the functor that associates to an analytic space $Y$ the set of morphism $f\mathrel{\mathop:}Y\rightarrow X$ of locally ringed spaces such that the following conditions hold:
\begin{itemize}
\item For every open affine $U=\Spec A$ in $X$ and every affinoid domain $V=\calM(\calB)$ in $Y$ with $f(V)\subseteq U$ the induced homomorphism $f^\#\mathrel{\mathop:}A\rightarrow\calB$ is continuous and bounded, i.e. we have $\vert f^\#(a)\vert\leq 1$ for every $a\in A$. 
\item For every point $y\in Y$ the induced homomorphism 
\begin{equation*}
\calO_{X,f(y)}\longrightarrow \calO_{Y,y}\longrightarrow \calH(y)
\end{equation*}
has values in $\calH(y)^\circ$ and therefore induces a local homomorphism $\calO_{X,\phi(y)}\rightarrow\calH(y)^\circ$. 
\end{itemize}

In \cite[Proposition et D\'efinition 1.3]{Thuillier_toroidal} Thuillier shows that this functor is representable by an analytic space $X^\beth$. The association $X\mapsto X^\beth$ defines a functor
\begin{equation*}
(.)^\beth\mathrel{\mathop:}\mathbf{Schemes}_{loc.f.t./k}\longrightarrow \mathbf{An}_k 
\end{equation*}
that respect fiber products and lands in $\mathbf{An}_k^{st}$. There is a natural morphism from $X^\beth$ to an analytic domain in the usual analytic space $X^{an}$ associated to $X$. If $X$ is separated over $k$, this morphism is injective, and, if $X$ is proper over $k$, we have $X^\beth=X^{an}$. 

\begin{example}\label{example_affinebeth}
Let $X=\Spec A$ be an affine scheme of finite type over $k$. The analytic space $X^{an}$ associated  to $X$ is the space of multiplicative seminorms on $A$ that restrict to the trivial norm on $k$ and $X^\beth$ is the affinoid domain in $X^{an}$ consisting of all seminorms that are bounded, i.e. that fulfill $\vert a\vert \leq 1$ for all $a\in A$. 
\end{example}

In general, one can construct $X^\beth$ by taking an open affine cover $U_i=\Spec A_i$ of $X$ and glueing the $U_i^\beth$ (as constructed in Example \ref{example_affinebeth}) over their intersections. So, in particular, we observe that $X^\beth$ is a strict analytic space.

\subsection{Stacky generic fibers}

The goal of this section is to generalize the functor $(.)^\beth$ to algebraic stacks locally of finite type over $k$ and to study its basic topological properties. In \cite{Ulirsch_trop=quot} the author has introduced a theory of geometric stacks over the category $\mathbf{An}_k$  of $k$-analytic spaces with the \emph{\'etale topology} (see \cite[Section 4]{Berkovich_etalecoho}), so called \emph{analytic stacks}. 

Unfortunately the functor $(.)^\beth$ does not send \'etale morphisms in the category of schemes to \'etale morphisms in $\mathbf{An}_k$ and therefore does not define a morphism between the respective \'etale sites. Therefore, in this article, we prefer to work with an alternative approach to non-Archimedean analytic stacks that can be found in \cite{Yu_Gromovcompactness, PortaYu_higherGAGA}. The theory essentially remains the same, but the base category is replaced by the category of strictly $k$-analytic spaces $\mathbf{An}_k^{st}$ with the $G$-\'etale topology $\tau_{G-et}$ (in the sense of \cite{ConradTemkin_descent}) or, in the terminology of \cite{Yu_Gromovcompactness, PortaYu_higherGAGA}, the quasi-\'etale topology. 

\begin{definition}[\cite{Yu_Gromovcompactness}*{Definition 6.2}]
A \emph{strict analytic stack} $\calX$ is a stack over $(\mathbf{An}^{st}_k)_{G-et}$ whose diagonal is representable (by \emph{strict analytic spaces}; defined in analogy with algebraic spaces) and that admits a surjective $G$-smooth morphism $U\rightarrow\calX$ from a strict analytic space $U$\footnote{In the category of strict analytic spaces the notions of $G$-smooth or $G$-\'etale morphisms (the terminology in \cite{ConradTemkin_descent}) are interchangable with the notions of \emph{quasi-smooth} or \emph{quasi-\'etale} (the terminology in \cite{Ducros_families}).}. 
\end{definition}

The $2$-category of strict analytic stacks is the full subcategory of the $2$-category of stacks over $(\mathbf{An}^{st}_k)_{G-et}$, whose objects are strict analytic stacks. The basic results concerning groupoid presentations and their quotients, as e.g. developed in \cite{Ulirsch_trop=quot}, immediately transfer to the framework of strict analytic stacks. In particular, given a $G$-smooth and surjective groupoid $\big(R\rightrightarrows U\big)$ in the category strict analytic spaces, we find that the quotient stack $\big[U\big/R\big]$ is a strict analytic stack.

In \cite{AbramovichCaporasoPayne_tropicalmoduli}*{Definition 6.1.2} the authors generalize the $(.)^\beth$-functor to separated algebraic spaces (whose analytification is an analytic space by \cite{ConradTemkin_algspaces}*{Theorem 1.2.1}). We want to generalize the $(.)^\beth$-functor to all algebraic stacks that are locally of finite type over $k$.

Given a smooth (or \'etale) morphism $f\colon X\rightarrow Y$ of schemes, the induced map $f^\beth\colon X^\beth\rightarrow Y^\beth$ is $G$-smooth (or $G$-\'etale respectively). Therefore the functor $(.)^\beth$ induces a morphism
\begin{equation*}
\beta\colon (\mathbf{An}^{st}_k,\tau_{G-et})\longrightarrow (\mathbf{Sch}_{loc.f.t./k}, \tau_{et})
\end{equation*}
of the respective sites that respects fiber products and sends smooth morphisms to $G$-smooth morphisms. Thus, we have the following Proposition \ref{prop_beth}.

\begin{proposition}\label{prop_beth}
Given an algebraic stack $\calX$, the pullback $\beta^\ast \calX$ is a strict analytic stack such that, whenever $\big[U\big/R\big]\simeq \calX$ is a smooth groupoid presentation of $\calX$ the stack $\beta^\ast\calX$ is a quotient of the induced groupoid $\big(R^\beth\rightrightarrows U^\beth\big)$. 
\end{proposition}

This a version of \cite{Ulirsch_trop=quot}*{Proposition 2.19} over the site $(\mathbf{An}_k^{st},\tau_{G-et})$. Its proof proceeds in complete analogy and we leave the details to the avid reader (also see \cite{CavalieriChanUlirschWise_tropstack}*{Proposition 1.11} for a general principle that applies to this situation). 

We set $\calX^\beth=\beta^\ast\calX$. By the general properties of pullbacks, the association $\calX\mapsto \beta^\ast\calX$ defines a lax $2$-functor 
\begin{equation*}
(.)^\beth\colon \mathbf{Alg.Stacks}_{loc.f.t./k}\longrightarrow \mathbf{An.Stacks}_k^{st}
\end{equation*}
from the $2$-category of algebraic stacks (locally of finite type over $k$) to the category of strict analytic stacks. 

Expanding on \cite{Ulirsch_trop=quot}*{Section 3.1}, we may functorially associate to a strict analytic stack $\calX$ its \emph{underlying topological space} $\vert \calX\vert$ such that whenever there is a $G$-smooth cover $U\rightarrow \calX$ the induced map $\vert U\vert \rightarrow \vert \calX\vert$ is a surjective quotient map.

An analogue of Proposition \cite[Proposition 3.8]{Ulirsch_trop=quot} for the $(.)^\beth$ functor shows that, if $\calX$ is an algebraic stack that is locally of finite type over $k$, we can describe the underlying topological space $\big\vert\calX^\beth\big\vert$ of $\calX^\beth$ as follows: 
\begin{itemize}
\item Its points are equivalence classes of pairs $(R,\phi)$ consisting of a valuation ring extending $k$ and a morphism $\phi\mathrel{\mathop:}\Spec R\rightarrow\calX$. 
\item Two such pairs $(R,\phi)$ and $(S,\psi)$ are hereby \emph{equivalent}, if there is a valuation ring $\calO$ extending both $R$ and $S$ making the diagram
\begin{equation*}\begin{CD}
\Spec(\calO)@>>>\Spec(S)\\
@VVV @VV\psi V\\
\Spec(R) @>\phi>> \calX
\end{CD}\end{equation*}
$2$-commutative. 
\end{itemize}


\section{Functorial tropicalization of logarithmic stacks}\label{section_functroplogstack}

In \cite{Ulirsch_functroplogsch} the author constructs a natural tropicalization maps associated to a (fine and saturated) logarithmic scheme. In this section we generalize this construction to the case of a fine and saturated logarithmic stack $\calX$ that is locally of finite type over $k$. In the case that $\calX$ is a logarithmically smooth Deligne-Mumford stack, this agrees with the non-Archimedean skeleton of $\calX$ constructed in \cite[Section 6]{AbramovichCaporasoPayne_tropicalmoduli}. 

Consider a (fine and saturated) monoid $P$. Denote by $\sigma_P=\Hom(P,\R_{\geq 0})$ the associated rational polyhedral cone and by $\sigmabar_P=\Hom(P,\Rbar_{\geq 0})$ its canonical compactification. It is well-known that $\sigmabar_{P}$ has a stratification by locally closed subsets that correspond to the points of $\Spec P$ (and, if $P$ is toric, to the torus orbits of $U_P=\Spec k[P]$) (see \cite{Ulirsch_functroplogsch, HuszarMarcusUlirsch_troplogclutch&glue} for details). Write $\chi^p$ for the character of $p\in P$ as a basis element of $k[P]=\bigoplus_{p\in P}k\cdot\chi^p$. There is a natural continuous tropicalization map 
\begin{equation*}\begin{split}
\trop_P\colon U_P^\beth&\longrightarrow\sigmabar_P\\
x&\longmapsto \big(p\mapsto -\log \vert \chi^p\vert_x\big)
\end{split}\end{equation*}
that is functorial with respect to homomorphisms of the monoid $P$ (or alternatively with respect to the toric morphisms of the affine toric variety $U_P=\Spec k[P]$). 

Recall from Definition \ref{definition_small} above that a logarithmic scheme $X$ is said to be \emph{small}, if the Artin fan $\calA_X$ of $X$ is an Artin cone $\calA_{P_X}$ for the monoid $P_X=\Gamma(X,\Mbar_X)$ and the induced isomorphism $P_X\xrightarrow{\sim}\Gamma(X,\Mbar_X)$ lifts to a chart of $X$. 
Given a small logarithmic scheme $X$ we write $\sigmabar_X$ for the canonical compactification $\sigmabar_{P_X}$ of $\sigma_{P_X}$.

\begin{proposition}\label{prop_smalltrop}
Let $X$ be a small logarithmic scheme. 
\begin{enumerate}[(i)]
\item There is a natural continuous tropicalization map 
\begin{equation*}
\trop_X\colon X^\beth \longrightarrow \sigmabar_{X}
\end{equation*}
such that, whenever $\gamma\colon P_X\rightarrow \calO_X$ is a lift of $P_X\xrightarrow{\sim}\Gamma(X,\Mbar_X)$ to a chart of $X$, the tropicalization map factors as 
\begin{equation*}
X^\beth \xlongrightarrow{\gamma^\sharp} U_{P_X}^\beth \xlongrightarrow{\trop_{P_X}} \sigmabar_X=\sigmabar_{P_X} \ .
\end{equation*}
\item If $f\colon X\rightarrow X'$ is a logarithmic morphism between two small logarithmic schemes, there is a morphism $\sigma(f)\colon \sigma_X\rightarrow \sigma_{X'}$ that extends to a continuous map $\sigmabar(f)\colon \sigmabar_X\rightarrow\sigmabar_{X'}$ that makes the diagram
\begin{equation*}\begin{CD}
X^\beth @>\trop_X>>\sigmabar_X\\
@Vf^\beth VV @VV\sigmabar(f)V\\
(X')^\beth @>\trop_{X'}>>\sigmabar_{X'}
\end{CD}\end{equation*}
commute. The association $f\mapsto\sigmabar(f)$ is functorial in $f$. 
\end{enumerate}
\end{proposition}

\begin{proof}
Consider two lifts $\gamma,\gamma'\colon P_X\rightarrow \calO_X$ of $P_X\xrightarrow{\sim}\Gamma(X,\Mbar_X)$ to a chart of $X$. For $p\in P_X$ we have $\gamma(p)=a \gamma'(p)$ for some unit $a\in\calO_X^\ast$ and therefore 
\begin{equation*}
\big\vert \gamma(p)\big\vert _x =\big\vert a \gamma'(p)\big\vert_x=\big\vert a\big\vert_x \big\vert \gamma'(a)\big\vert_x = \big\vert \gamma'(a)\big\vert_x \ ,
\end{equation*}
since $\vert a\vert_x=1$ for all $a\in \calO_X^\ast$, as in turn $\vert a\vert_x\leq 1$ and $\vert a^{-1}\vert_x\leq 1$. Therefore we also have $\trop_P\circ(\gamma^\sharp)^\beth=\trop_P\circ \big((\gamma')^\sharp\big)^\beth$ and this shows that the continuous tropicalization $\trop_X$ is well-defined, which proves part (i). 

For part (ii), a logarithmic morphism $f\colon X\rightarrow X'$ induces a monoid homomorphism $f^\flat\colon P_{X'}\rightarrow P_X$, which in turn induces a cone map $\sigma(f)\colon\sigma_X\rightarrow\sigma_{X'}$ that extends to a continuous map $\sigmabar(f)\colon \sigmabar_X\rightarrow \sigmabar_{X'}$ on the canonical extensions. Choosing lifts of $P_X$ and $P_{X'}$ to charts of $X$ and $X'$ respectively, we obtain a diagram 
\begin{equation*}\begin{CD}
\calO_X @<<< P_X\\
@AAA @AAA\\
\calO_{X'}@<<< P_{X'}
\end{CD}\end{equation*}
that commutes up to multiplication with a unit in $\calO_X^\ast$. By the above reasoning the tropicalization map is invariant under multiplication by units and therefore the induced diagram
\begin{equation*}\begin{CD}
X^\beth @>\trop_X>>\sigmabar_X\\
@Vf^\beth VV @VV\sigmabar(f)V\\
(X')^\beth @>\trop_{X'}>>\sigmabar_{X'}
\end{CD}\end{equation*}
commutes. The functoriality of $f\mapsto \sigmabar(f)$ is an immediate consequence of its definition.
\end{proof}

\begin{remark}
The construction of $\trop_X$ in Proposition \ref{prop_smalltrop} is really a version of the \emph{local tropicalization map} of Popescu-Pampu and Stepanov \cite{PopescuPampuStepanov_localtrop}. When $X=\Spec A$ we may describe it as the map $X^\beth\rightarrow \sigmabar_X=\Hom(P_X,\Rbar_{\geq 0})$ given by
\begin{equation*}
x\longmapsto \big(p\mapsto -\log \vert \gamma(p)\vert_x\big) 
\end{equation*} 
for $x\in X^\beth$.
\end{remark}

 Let $\Sigma$ be a generalized cone complex. Recall from \cite{AbramovichCaporasoPayne_tropicalmoduli}*{Section 2.6} (also see \cite{Ulirsch_functroplogsch}*{Section 3.3 and 3.5} that the \emph{canonical extension} $\Sigmabar$ of $\Sigma$ is given as the topological colimit of the induced diagram of extended cones $\sigmabar_\alpha$.   Any morphism $f\colon \Sigma\rightarrow\Sigma'$ of generalized cone complexes canonically extends to a continuous map $\overline{f}\colon\Sigmabar\rightarrow \Sigmabar'$ of the associated canonical extensions.

\begin{proposition}\label{prop_functroplogstack} Let $\calX$ be a fine and saturated logarithmic stack locally of finite type over $k$. Denote by $\calF_\calX$ the Kato stack associated to the Artin fan $\calA_\calX$ associated to $\calX$ and write $\Sigma_\calX$ for the generalized cone complex associated to $\calF_\calX$ as well as $\Sigmabar_{\calX}$ for its canonical extension. 
\begin{enumerate}[(i)]
\item There is a unique natural continuous \emph{tropicalization map} 
\begin{equation*}
\trop_\calX\colon\big\vert\calX^\beth\big\vert \longrightarrow \Sigmabar_\calX
\end{equation*}
such that, whenever $U\rightarrow \calX$ is a strict smooth morphism from a small logarithmic scheme $U$ to $\calX$, the induced diagram 
\begin{equation*}\begin{CD}
U^\beth @>\trop_U>>\sigmabar_{U}\\
@VVV @VVV\\
\big\vert\calX^\beth\big\vert @>\trop_\calX>>\Sigmabar_\calX 
\end{CD}\end{equation*}
commutes.
\item A logarithmic morphism $f\colon \calX\rightarrow \calX'$ induces a morphism $\Sigma(f)\colon \Sigma_\calX\rightarrow \Sigma_{\calX'}$ of generalized cone complexes that extends to a continuous map $\Sigmabar(f)\colon \Sigmabar_\calX\rightarrow\Sigmabar_{\calX'}$ which makes the diagram
\begin{equation*}\begin{CD}
\big\vert\calX^\beth \big\vert @>\trop_\calX>>\Sigmabar_\calX\\
@V\vert f^\beth\vert VV @VV\Sigmabar(f)V\\
\big\vert (\calX')^\beth\big\vert @>\trop_{\calX'}>>\Sigmabar_{\calX'}
\end{CD}\end{equation*}
commute. 
\item If $\calX$ is a logarithmically smooth Deligne-Mumford stack, the tropicalization map has a section 
\begin{equation*}
J_\calX\colon \Sigmabar_\calX\longrightarrow \big\vert \calX^\beth\big\vert
\end{equation*}
such that the composition 
\begin{equation*}
\bfp_\calX=J_\calX\circ \trop_\calX\colon \big\vert\calX^\beth\big\vert\longrightarrow \big\vert \calX^\beth\big\vert
\end{equation*}
is a strong deformation retraction onto the non-Archimedean skeleton of $\calX^\beth$ (in the sense of \cite{Thuillier_toroidal, AbramovichCaporasoPayne_tropicalmoduli}). 
\end{enumerate}
\end{proposition}

\begin{proof}
Choose a strict smooth  cover of $\calX$ by a disjoint union $U=\bigsqcup_i U_i$ of a small logarithmic schemes. The fiber product $R=U\times_\calX U$ may not be small, but can be covered a disjoint union $V=\bigsqcup_j V_j$ of small logarithmic schemes. By \cite{Ulirsch_trop=quot}*{Proposition 3.2 (ii) and Proposition 3.4 (ii)} (or rather their analogues for strict analytic stacks, proved in the exact same way) the underlying topological space $\big\vert \calX^\beth\big\vert$ is the colimit of $\big\vert V^\beth\big\vert\rightrightarrows  \big\vert U^\beth\big\vert$.

Let $\{\Spec P_i\rightarrow \calF_\calX\}_{i\in I}$ be a collection of strict morphisms that cover $\calF_\calX$. Again, setting $F'=\bigsqcup \Spec P_i$, the fiber product $F''=F'\times_\calF F'$ may not be affine, but we can find a strict cover $\{\Spec Q_j\rightarrow F''\}_{j\in J}$. Since every strict morphism $\Spec P\rightarrow \calF_\calX$ must factor through at least one of the $\Spec P_i\rightarrow \calF_\calX$, we have that $\Sigmabar_\calX$ is in fact equal to the topological colimit $\Sigmabar_{F''}\rightrightarrows \Sigmabar_{F'}$.

We may now take $\{\Spec P_i\rightarrow \calF_{\calX}\}_{i\in I}$ to be the maps induced from $\{U_i\rightarrow \calX\}_{i\in I}$ and $\Spec Q_j$ to be the maps $\{\Spec Q_j\rightarrow F'\}_{j\in J}$ induced from $\{V_j\rightarrow R\}_{j\in J}$. By Proposition \ref{prop_smalltrop} (ii) and the universal property of colimits, the induced collection of tropicalization maps $\trop_{U_i}$ and $\trop_{V_j}$ descends to a well-defined continuous tropicalization map $\trop_\calX\colon \big\vert\calX^\beth\big\vert\rightarrow \Sigmabar_\calX$. 

Now let $\widetilde{U}=\bigsqcup_{k\in K} \widetilde{U}_k\rightarrow \calX$ be a strict smooth cover of $\calX$ by small logarithmic schemes $\{\widetilde{U}_k\}_{k\in K}$ that refines $\{U_i\}_{i\in I}$  and $V=\bigsqcup_{l\in L} V_l\rightarrow \widetilde{R}$ be a strict smooth cover of $\widetilde{R}=\widetilde{U}\times_\calX \widetilde{U}$ by small logarithmic schemes $\{\widetilde{V}_l\}_{l\in L}$ that refines $\{V_k\}_{j\in J}$. Then Proposition \ref{prop_smalltrop} (ii) shows that the induced diagram of tropicalization maps commutes and therefore they descend to the same tropicalization map 
\begin{equation*}\big\vert \calX^\beth\big\vert\longrightarrow \Sigmabar_\calX \ .
\end{equation*}

In general, given two strict smooth covers of $\calX$ by small logarithmic schemes, there is common refinement as above and the above reasoning shows that the construction of $\trop_\calX$ does not depend on the choice of the covers. Every morphism $U\rightarrow\calX$ is part of some strict smooth cover by small logarithmic schemes and therefore the diagram 
\begin{equation*}\begin{CD}
U^\beth @>\trop_U>>\sigmabar_{U}\\
@VVV @VVV\\
\big\vert\calX^\beth\big\vert @>\trop_\calX>>\Sigmabar_\calX 
\end{CD}\end{equation*}
commutes. This finishes the proof of (i). 

 For part (ii), let $f\colon \calX\rightarrow \calX'$ be a logarithmic morphism. Choose a strict smooth cover $U'=\bigsqcup_{i'} U_{i'}'$ of $\calX'$ by small logarithmic schemes $U_{i'}'$ and a strict smooth cover $V'=\bigsqcup_{j'} V_{j'}'$ of $R'=U'\times_{\calX'} U'$ by small logarithmic schemes $V_{j'}'$ as above, as well as a strict smooth cover $U=\bigsqcup_i U_i$ of $\calX$ by small logarithmic schemes $U_i$ and a strict smooth cover $V=\bigsqcup_j V_j$ of $R=U\times_\calX U$ by small logarithmic schemes $V_j$ such that $f$ naturally restricts to morphisms $U_i\rightarrow U_{i'}'$ and $V_j\rightarrow V_{j'}'$ for every $i$ and $j$ and some $i'$ and $j'$. Applying Proposition \ref{prop_smalltrop} to the diagram of such maps yields a diagram of morphisms of rational polyhedral cones that descends to a continuous map $\big\vert\Sigma(f)\big\vert\colon\vert\Sigma_\calX\vert\rightarrow\vert\Sigma_{\calX'}\vert$. This association is functorial in $f$. 

Let $U\rightarrow \calX$ be a smooth strict morphism and let $x$ be a geometric point in the deepest stratum of $U$. Choose a small smooth neighborhood $U'\rightarrow \calX'$ of $f(x)$. We can shrink $U$ such that $f$ restricts to $U\rightarrow U'$. By Proposition \ref{prop_smalltrop} this induces a map $\sigma_U\rightarrow\sigma_{U'}$ that makes the square 
\begin{center}\begin{tikzcd}
\sigma_U \arrow[r]\arrow[d]&\sigma_{U'}\arrow[d]\\
\Sigma_\calX\arrow[r]&\Sigma_{\calX'}
\end{tikzcd}\end{center}
commute and so $\big\vert \Sigma(f)\big\vert$ is a morphism of generalized cone complexes. By the definition of $\trop_\calX$, we may check the commutativity of the square
\begin{equation*}\begin{CD}
\big\vert\calX^\beth \big\vert @>\trop_\calX>>\Sigmabar_\calX\\
@V\vert f^\beth\vert VV @VV\Sigmabar(f)V\\
\big\vert (\calX')^\beth\big\vert @>\trop_{\calX'}>>\Sigmabar_{\calX'}
\end{CD}\end{equation*}
\'etale locally on $\calX$ and then this is precisely Proposition \ref{prop_smalltrop} (ii).

Finally, for part (iii), we recall that the case of a small logarithmically smooth scheme has been treated in \cite{Thuillier_toroidal}*{Corollaire 3.13} and the general case follows by descent along colimits (see \cite{Thuillier_toroidal}*{Proposition 3.29 and Corollaire 3.30} as well as \cite{AbramovichCaporasoPayne_tropicalmoduli}*{Proposition 6.1.4}). We remark that in \cite{AbramovichCaporasoPayne_tropicalmoduli} the authors work with the coarse moduli space $X$ of $\calX$, but, by \cite{Ulirsch_trop=quot}*{Proposition 3.8} the underlying topological space $\big\vert\calX^\beth\big\vert$ is naturally homeomorphic to $\big\vert X^\beth\big\vert$. 
\end{proof}

\begin{example}[Logarithmic schemes without monodromy]
Suppose that $X$ is a Zariski logarithmic scheme without monodromy. By Example \ref{example_logschemeswithoutmonodromy} above there is a natural characteristic morphism $\phibar_X\colon X\rightarrow F_X$ in a Kato fan. The generalized cone complex $\Sigma_X$ of $X$ is the cone complex that is given as the $\R_{\geq 0}$-valued points of $F_X$ and its canonical extension as the $\Rbar_{\geq 0}$-valued points of $F_X$ (see \cite{Ulirsch_functroplogsch}*{Section 3} for details). 

Following \cite[Section 6.1]{Ulirsch_functroplogsch}  the tropicalization map $\trop_X\mathrel{\mathop:}X^\beth\rightarrow \Sigmabar_X$ into the \emph{extended cone complex} $\Sigmabar_X=\Sigmabar_{F_X}=F_X(\Rbar_{\geq 0})$ associated to $X$ is defined as follows: A point $x\in X^\beth$ can be represented by a morphism $\underline{x}\mathrel{\mathop:}\Spec R\rightarrow (X,\calObar_{X})$, where $\calObar_X$ denotes the sharp monoid sheaf $\calO_X/\calO_X^\ast$ on $X$; its image $\trop_X(x)$ in $\Sigmabar_X$ is given by the composition
\begin{equation*}\begin{CD}
\Spec \Rbar_{\geq 0}@>\val^\#>>\Spec R@>\underline{x}>> (X,\calObar_X)@>\phibar_X >>F_X \ ,
\end{CD}\end{equation*}
where $\val^\#$ is the morphism of Kato fans induced by the valuation $\val\mathrel{\mathop:}R\rightarrow\Rbar_{\geq 0}$ on $R$. It has been shown in \cite{Ulirsch_functroplogsch} that $\trop_X$ is well-defined, continuous, and functorial with respect to logarithmic morphisms. 
\end{example}


\section{Tropicalization and analytification} \label{section_trop=anal}
The purpose of this section is to prove Theorem \ref{thm_trop=anal} from the introduction. In the following Proposition \ref{prop_trop=quot} we first consider the case of an affine toric variety $U_P=\Spec k[P]$, defined by a fine and saturated monoid $P$, with its natural morphism $\phi_P\colon=\phi_{U_P}\colon U_P\rightarrow \calA_P$ to the Artin cone $\calA_P=\big[U_P\big/T\big]$ with big torus $T=\Spec k[P^{gp}]$. 

\begin{proposition}\label{prop_trop=quot}
Let $P$ be a fine and saturated monoid. Then there is a natural homeomorphism $\mu_P\mathrel{\mathop:}\big\vert\calA_P^\beth\big\vert\xrightarrow{\sim}\sigmabar_P$ that makes the diagram
\begin{center}\begin{tikzcd}
U_P^\beth \arrow[rd,"\phi_P^\beth"'] \arrow[rrd,bend left,"\trop_P"]&&\\
& \big\vert\calA_P^\beth\big\vert \arrow[r,"\sim","\mu_P"'] & \sigmabar_P
\end{tikzcd}\end{center}
commute.
\end{proposition}

Proposition \ref{prop_trop=quot} is essentially a special case of \cite[Theorem 1.1]{Ulirsch_trop=quot} and its proof  (provided below) is a variation of the proof in \cite{Ulirsch_trop=quot}. 

By \cite{Ulirsch_functroplogsch}*{Lemma 2.1 (i)} we may choose a toric submonoid $\widetilde{P}$ of $P$ such that $\widetilde{P}\oplus P^{tors}=P$. Since there are natural isomorphisms $\calA_P\simeq \calA_{\widetilde{P}}$ and $\sigma_P\simeq \sigma_{\widetilde{P}}$ and both the tropicalization map $\trop_P$ and the quotient map $U_P\rightarrow \calA_P$ factor as $U_P^\beth\rightarrow U_{\widetilde{P}}^\beth\rightarrow \sigmabar_P$ and 
$U_P\rightarrow U_{\widetilde{P}}\rightarrow \calA_P$ respectively, we may assume in the proof of Proposition \ref{prop_trop=quot} that $P$ is toric. 

We begin by recalling the construction of the non-Archimedean skeleton of the analytic space $U_P^\beth$ associated to the toric variety $U_P$, due to Thuillier \cite[Section 2]{Thuillier_toroidal}. Write $\chi^p$ for the character of $p\in P$ as a basis element of $k[P]=\bigoplus_{p\in P}k\cdot\chi^p$ and consider the natural continuous tropicalization map 
\begin{equation*}\begin{split}
\trop_P\mathrel{\mathop:}U_P^\beth&\longrightarrow\sigmabar_P=\Hom(P,\Rbar_{\geq 0})\\
x&\longmapsto \Big(p\mapsto -\log\vert\chi^p\vert_x\Big) 
\end{split}\end{equation*}
essentially as introduced in \cite{Kajiwara_troptoric} and \cite{Payne_anallimittrop}. It has a natural continuous section $J_P\mathrel{\mathop:}\sigmabar_P\longrightarrow U_P^\beth$ defined by associating to $u\in \sigmabar_P=\Hom(P,\Rbar_{\geq 0})$ the seminorm $J_P(u)$ on $k[P]$, which is given by
\begin{equation*}
J_P(u)(f)=\max_{p\in P}\vert a_p\vert \exp\big(-u(p)\big) 
\end{equation*}
for $f=\sum_{p\in P}a_p\chi^p\in k[P]$. The composition $\bfp_P=J_P\circ\trop_P\mathrel{\mathop:}U_P^\beth\rightarrow U_P^\beth$ is a continuous retraction map and given by associating to $x\in U_P^\beth$ the seminorm $\bfp_P(x)$ that is determined by
\begin{equation*}
\bfp_P(x)(f)=\max_{p\in P}\vert a_p\vert \vert\chi^p\vert_x
\end{equation*}
for $f=\sum_{p\in P}a_p\chi^p\in k[P]$. Its image $\frakS(U_P)$ is called the \emph{non-Archimedean skeleton} of $U_P^\beth$ and is naturally homeomorphic to $\sigmabar_P$ via $J_P$. Moreover, the results of \cite[Section 2.2]{Thuillier_toroidal} show that $\bfp_P$ is in fact a strong deformation retraction.

Denote by $\mu\mathrel{\mathop:}T\times U_P\rightarrow U_P$ the operation of $T=\Spec k[M]$, with $M=P^{gp}$, on $U_P$ and note that this morphism is induced by the homomorphism
\begin{equation*}\begin{split}
\mu^\#\mathrel{\mathop:} k[P]&\longrightarrow k[M]\otimes_kk[P]\\
\chi^p&\longmapsto \chi^p\otimes\chi^p \ .
\end{split}\end{equation*}
Moreover consider the second projection map $\pi\mathrel{\mathop:}T\times U_P\rightarrow U_P$ induced by the homomorphism
\begin{equation*}\begin{split}
\pi^\#\mathrel{\mathop:} k[P]&\longrightarrow k[M]\otimes_kk[P]\\
\chi^p&\longmapsto 1\otimes\chi^p \ .
\end{split}\end{equation*}

\begin{lemma}\label{lemma_otimeshat}
Let $x\in U_P^\beth$ and consider the point $\eta\hat{\otimes}x\in T^\beth\times U_P^\beth$ given by the seminorm
\begin{equation*}
\vert f\vert_{\eta\hat{\otimes}x}=\max_{m\in M}\vert a_m\vert \vert f_m\vert_x
\end{equation*}
for an element $f=\sum_{m\in M}a_m \chi^m\otimes f_m\in k[M]\otimes_kk[P]$ with unique regular functions $f_m\in k[P]$. Then we have
\begin{equation*}
\pi^\beth(\eta\hat{\otimes}x)=x
\end{equation*}
as well as 
\begin{equation*}
\mu^\beth(\eta\hat{\otimes}x)=\bfp_P(x) \ .
\end{equation*}
\end{lemma}

\begin{proof}
Let $f=\sum_{p\in P}a_p\chi^p\in k[P]$. Then we have
\begin{equation*}
\vert f\vert_{\pi^\beth(\eta\hat{\otimes}x)}=\Big\vert\sum_{p\in P}a_p1\otimes\chi^p\Big\vert_{\eta\hat{\otimes}x}=\vert1\otimes f\vert_{\eta\hat{\otimes}x}=\vert f\vert_x
\end{equation*}
as well as 
\begin{equation*}
\vert f\vert_{\mu^\beth(\eta\hat{\otimes}x)}=\Big\vert\sum_{p\in P}a_p\chi^p\otimes\chi^p\Big\vert_{\eta\hat{\otimes}x}=\max_{p\in P}\vert a_p \vert \vert\chi^p\vert_x=\vert f\vert_{\bfp_P(x)}
\end{equation*}
and this implies the claim.
\end{proof}

\begin{proof}[Proof of Proposition \ref{prop_trop=quot}]
The topological space $\big\vert \big[U_P^\beth\big/T^\beth\big]\big\vert$ is the topological colimit of the two maps
\begin{equation}\label{eq_tropgroupoid}
(\pi^\beth,\mu^\beth\mathrel{\mathop:}T^\beth\times U_P^\beth\rightrightarrows U_P^\beth) \ .
\end{equation}
Therefore it is enough to show that the deformation retraction $\bfp_P\mathrel{\mathop:}U_P^\beth\rightarrow\frakS(U_P)$ makes the skeleton $\frakS(U_P)$ into a topological colimit of the above two maps. 
\begin{itemize}
\item Let $x, x'\in U_P^\beth$ and $y\in T^\beth\times U_P^\beth$ with $\pi^\beth(y)=x$ and $\mu^\beth(y)=x'$. In this case we have $\bfp_P(x)=\bfp_P(x')$, since 
\begin{equation*}\begin{split}
\big\vert\chi^{p}\big\vert_{x'}&=\big\vert \chi^p\big\vert_{\mu^\beth(y)}=\big\vert\chi^p\otimes\chi^p\big\vert_y\\&=\big\vert \chi^p\otimes 1\big\vert_y \cdot\big\vert 1\otimes \chi^p\big\vert_y=\big\vert1\otimes\chi^p\big\vert_y\\&=\big\vert\chi^p\big\vert_{\pi^\beth(y)}=\big\vert \chi^p\big\vert_x
\end{split}\end{equation*}
for all $p\in P$ using the fact that $\big\vert\chi^m\otimes 1\big\vert_y=1$ for all $m\in M$.
\item For $x\in U_P^\beth$ Lemma \ref{lemma_otimeshat} implies that there is a point $y=\eta\hat{\otimes}x\in T^\beth\otimes U_P^\beth$ such that $\pi^\beth(y)=x$ and $\mu^\beth(y)=\bfp_P(x)$. So, given two points $x,x'\in U_P^\beth$ such that $\bfp_P(x)=\bfp_P(x')$, their image in $\big\vert\big[U_P^\beth\big/T^\beth\big]\big\vert$ is equal as well. 
\end{itemize}

That is why the skeleton $\frakS(U_P)$ is the set-theoretic colimit of \eqref{eq_tropgroupoid} and, since $\bfp_P$ is continuous and proper, this is actually a colimit in the category of topological spaces. 
\end{proof}

\begin{proposition}\label{prop_anal=conecomplex}
Given a Kato stack $\calF$, there is a natural homeomorphism
\begin{equation*}
\mu_\calF\colon \big\vert\calA_{\calF}^\beth\big\vert \xlongrightarrow{\sim}\Sigmabar_{\calF} \ .
\end{equation*}
\end{proposition}

\begin{proof}
The generalized extended cone complex of $\calF$ is the defined to be the colimit of the diagram of all $\sigmabar_P$ associated to the combinatorial cone stack $\calF^{comb}$, i.e. the colimit over extended cones $\sigmabar_P$ associated to strict morphisms $\Spec P\rightarrow \calF$ with face maps induced from strict morphisms over $\calF$. By \cite{Ulirsch_trop=quot}*{Lemma 3.2 (ii) and Proposition 3.4 (ii)} the topological space $\big\vert \calA_\calF^\beth \big\vert$ is the colimit of a the diagram whose objects are the topological spaces $\big\vert\calA_P^\beth\big\vert$ indexed by the collection of strict morphisms $\Spec P\rightarrow \calF$ and whose arrows are induced by strict morphisms over $\calF$. Therefore Proposition \ref{prop_trop=quot} implies the claim.
\end{proof}

We now conclude this section and this article with the proof of Theorem \ref{thm_trop=anal}.

\begin{proof}[Proof of Theorem \ref{thm_trop=anal}]
By Proposition \ref{prop_anal=conecomplex} we have a natural homeomorphism between $\Sigmabar_\calX$ and the topological space $\big\vert \calA_{\calX}^\beth\big\vert$. It remains to show that the diagram
\begin{center}\begin{tikzcd}
\big\vert\calX^\beth\big\vert \arrow[rd,"\phi_\calX^\beth"']\arrow[rrd,bend left,"\trop_\calX"]&&\\
& \big\vert\calA_\calX^\beth\big\vert \arrow[r,"\sim","\mu_\calX"'] & \Sigmabar_\calX
\end{tikzcd}\end{center}
commutes. But this can be checked on strict small smooth open neighborhoods of $\calX$ and so we may assume that $\calX$ is representable by a small logarithmic  scheme $X$. Choose a lift of $P_X\xrightarrow{\sim}\Gamma(X,\Mbar_X)$ to a chart $\phi\colon P_X\rightarrow \calO_X$ of $X$. Then by Proposition \ref{prop_smalltrop} (i), the tropicalization map $\trop_X$ factors as 
\begin{equation*}
\trop_X\colon X^\beth \xrightarrow{\phi^\sharp} U_{P_X}^\beth \xrightarrow{\trop_{P_X}}\sigmabar_X=\sigmabar_{P_X}
\end{equation*}
and therefore, in this case, Proposition \ref{prop_trop=quot} implies the claim. 
\end{proof}




\bibliographystyle{alpha}
\bibliography{biblio}{}

\end{document}